\documentclass[a4paper, 11pt]{amsart}
\usepackage[numbers,square]{natbib}
\usepackage[english]{babel}
\usepackage{hyperref}
\usepackage{setspace}
\usepackage{bbm}
\usepackage[dvips]{graphicx}
\usepackage{lmodern} 
\usepackage{paralist}
\usepackage{float}
\usepackage{color}
\usepackage[left=2cm, right=2cm]{geometry}
\usepackage{amsmath}
\usepackage{amssymb}
\newcommand{\INT}[2]{\int\limits_{#1}^{#2} } 
\renewcommand{\epsilon}{\varepsilon}
\newcommand{\eps}{\epsilon}
\newcommand{\POI}[1]{\mathrm{Poi}\left(#1\right)}
\renewcommand{\theta}{\vartheta}
\renewcommand{\phi}{\varphi}
\renewcommand{\leq}{\leqslant}
\renewcommand{\geq}{\geqslant}
\newcommand{\D}{\ \mathrm{d}}
\newcommand{\R}{\mathbb{R}}
\newcommand{\CN}{\mathbb{C}} 
\newcommand{\I}{\mathrm{i}}
\newcommand{\1}{\mathbbm{1}}
\newcommand{\ind}{\mathbbm{1}} 

\renewcommand{\d}{\mathrm{d}}

\newcommand{\eee}{{\rm e}}
\newcommand{\e}{\mathrm{e}}

\newcommand{\BP}[2]{\P{#1\mid #2}}

\newcommand{\EXP}[1]{\exp\left(#1\right)}
\newcommand{\bD}{\mathbb{D}}

\DeclareMathOperator{\VAR}{Var}

\DeclareMathOperator{\RE}{Re}

\renewcommand{\l}{\left}
\renewcommand{\r}{\right}

\newcommand{\BINOMIALVERT}[2]{\mathrm{Bin}(#1,#2)}
\newcommand{\NBINOMIALVERT}[2]{\mathrm{NBin}(#1,#2)}

\newcommand{\todistr}{\overset{\rm d}{\underset{n\to\infty}\longrightarrow}}
\newcommand{\toweak}{\overset{\rm w}{\underset{n\to\infty}\longrightarrow}}
\newcommand{\toweaknon}{\overset{\rm w}{\underset{}\longrightarrow}}
\newcommand{\ton}{\overset{}{\underset{n\to\infty}\longrightarrow}}

\DeclareMathOperator{\argmax}{argmax}

\usepackage{amsthm}
\theoremstyle{plain}
\newtheorem{theorem}{Theorem}[section]
\newtheorem{lemma}[theorem]{Lemma}
\newtheorem{corollary}[theorem]{Corollary}

\theoremstyle{definition}

\theoremstyle{remark}
\newtheorem{remark}[theorem]{Remark}

\renewcommand{\P}[1]{\mathbb{P}\left[#1\right]} 

\newcommand{\N}{\mathbb{N}}
\newcommand{\E}{\mathbb{E}}

\newcommand{\IND}[1]{\1_{#1}}
\newcommand{\EW}[1]{\E\left[#1\right]}

\DeclareMathOperator{\IM}{Im}
\DeclareMathOperator{\COV}{Cov}

\newcommand{\Z}{\mathbb{Z}}

\newcommand{\FAF}{\textrm{FAF}}
\newcommand{\HAF}{\textrm{HAF}}
\newcommand{\SP}{\textrm{SP}}
\newcommand{\WP}{\textrm{WP}}

\begin{document}

\author{Hendrik Flasche}
\address{Institut f\"ur Mathematische Stochastik,
Westf\"alische Wilhelms-Universit\"at M\"unster,
Orl\'eans--Ring 10,
48149 M\"unster, Germany}
\email{hendrik.flasche@web.de}

\author{Zakhar Kabluchko}
\address{Institut f\"ur Mathematische Stochastik,
Westf\"alische Wilhelms-Universit\"at M\"unster,
Orl\'eans--Ring 10,
48149 M\"unster, Germany}
\email{zakhar.kabluchko@uni-muenster.de}

\title[Real Zeros of Random Analytic Functions]{Real zeroes of random analytic functions associated with geometries of constant curvature}



\begin{abstract}
Let $\xi_0,\xi_1,\ldots$ be i.i.d.\ random variables with zero mean and unit variance. We study the following four families of random analytic functions:
$$
P_n(z)
:=
\begin{cases}
\sum_{k=0}^n \sqrt{\binom nk} \xi_k z^k &\text{ (spherical polynomials)},\\
\sum_{k=0}^\infty \sqrt{\frac{n^k}{k!}} \xi_k z^k &\text{ (flat random analytic function)},\\
\sum_{k=0}^\infty \sqrt{\binom {n+k-1} k} \xi_k z^k &\text{ (hyperbolic random analytic functions)},\\
\sum_{k=0}^n \sqrt{\frac{n^k}{k!}} \xi_k z^k &\text{ (Weyl polynomials)}.
\end{cases}
$$
We compute explicitly the limiting mean density of real zeroes of these random functions. More precisely, we provide a formula for  $\lim_{n\to\infty} n^{-1/2}\mathbb E N_n[a,b]$, where $N_n[a,b]$ is the number of zeroes of $P_n$ in the interval $[a,b]$.
\end{abstract}

\maketitle

\section{Introduction and statement of results}
\subsection{Introduction}
Let $\xi_0,\xi_1,\ldots$ be independent, identically distributed  (i.i.d.)\ random variables with real values. Consider random polynomials of the form
\begin{equation*}
Q_n(z) = \sum_{k=0}^n \xi_k z^k, \quad z\in\CN.
\end{equation*}
The study of real zeroes of random polynomials began with the works of~\citet{bloch_polya} and~\citet{littlewood_offord1,littlewood_offord2,littlewood_offord3}.  Assuming that the $\xi_k$'s are standard normal, Kac~\cite{kac_asympt} derived an explicit formula for the expected number of real zeroes of $Q_n$  and proved that this number is asymptotically equivalent to $\frac 2 \pi \log n$, as $n\to\infty$. Ibragimov and Maslova~\cite{ibragimov_maslova1} proved that the same asymptotics continues to hold if the $\xi_k$'s have zero mean and finite second moment. Further results on the number of real zeroes, including an asymptotic formula for the variance and a central limit theorem, were obtained in the subsequent works by Ibragimov and Maslova~\cite{ibragimov_maslova2,maslova_variance,maslova_distribution}. For more recent results on the number of real roots, see~\cite{do_nguyen_vu_repulsion,nguyen_nguyen_vu,do_nguyen_vu_arbitrary}.

The asymptotic distribution of complex zeroes of $Q_n$ is also well-understood. Roughly speaking, most complex zeroes cluster near the unit circle and their arguments have approximately uniform distribution, as $n\to\infty$. More precisely, if we assign to each complex zero of $Q_n$ the weight $1/n$, then by a result of Ibragimov and Zaporozhets~\cite{ibragimov_zaporozhets}, the probability that the resulting random probability measure converges weakly to the uniform distribution on the unit circle equals $1$ if and only if $\E \log(1+|\xi_0|)<\infty$.

\subsection{Four families of random analytic functions}\label{subsec:four_families}
Along with polynomials whose coefficients are i.i.d.\ random variables as above, several other ensembles of random polynomials (or, more
generally, random analytic functions) appeared in the literature. These random analytic functions  have the form
\begin{equation}\label{eq:defrandompolynomial}
P_n(z):=\sum_{k=0}^\infty f_{n,k} \xi_kz^k,
\end{equation}
where $z\in\CN$ is a complex variable,  $(f_{n,k})_{n\in\N, k\in\N_0}$ are real deterministic coefficients to be specified below, and  $\xi_0,\xi_1,\ldots$ are i.i.d.\ real-valued random variables.
The following special cases of~\eqref{eq:defrandompolynomial} proved to be especially interesting:
\begin{equation}\label{eq:thefourcases}
f_{n,k}=
\begin{cases}
\sqrt{\binom{n}{k}}\IND{\{k\leq n\}}, & \textrm{binomial, elliptic, or spherical polynomials (\SP),} \\
\sqrt{\frac{n^k}{k!}}, & \textrm{flat random analytic function (\FAF),} \\
\sqrt{\binom{n+k-1}{k}}, & \textrm{hyperbolic random analytic function (\HAF),} \\
\sqrt{\frac{n^k}{k!}}\IND{\{k\leq n\}}, & \textrm{Weyl polynomials (\WP).}
\end{cases}
\end{equation}
The first three ensembles appeared in the theoretical physics literature as quantum chaotic eigenstates~\cite{bogomolny1,bogomolny2,hannay1,hannay2,forrester_honner,bleher_ridzal}, in the computational complexity literature~\cite{shub_smale}, see also~\cite{kostlan,peres_book,sodin_tsirelson,schehr_majumdar},  and are referred to as the \textit{elliptic} (or $SU(2)$), \textit{flat} (or $ISO(2)$), and \textit{hyperbolic} (or $SU(1,1)$) ensembles, respectively.  If the coefficients $\xi_k$ have the complex Gaussian distribution, the zero sets of these ensembles are invariant with respect to the isometries of the elliptic (or spherical)/flat/hyperbolic two-dimensional geometry~\cite{sodin_tsirelson,peres_book}.

Regarding the complex zeroes of these random analytic functions in the case of non-Gaussian coefficients, the following is known~\cite{kabluchko_zaporozhets}. Let $\mathcal{M}(\CN)$ be the space of locally finite measures on $\CN$ endowed with the vague topology.  Let $\mu_n$ be the point process of complex zeroes of $P_n$, that is $\mu_n$ is a random element of $\mathcal{M}(\CN)$ given by
$$
\mu_n :=  \sum_{z\in \CN\colon P_n(z)=0}\delta_z,
$$
where $\delta_z$ denotes the unit point mass at $z$. It was shown in~\cite{kabluchko_zaporozhets} that under the moment assumption $\E \log (1+|\xi_0|)<\infty$, the sequence $\frac 1n \mu_n$ converges weakly in probability on $\mathcal{M}(\CN)$ to the deterministic measure with Lebesgue density
$$
\rho(z) =
\begin{cases}
\pi^{-1}  (1+|z|^2)^{-2} &\textrm{in the \SP\ case, }\\
\pi^{-1} & \textrm{in the \FAF\ case, }\\
\pi^{-1} (1-|z|^2)^{-2}\ind_{\{|z| < 1\}} &\textrm{in the \HAF\ case, }\\
\pi^{-1}\ind_{\{|z| < 1\}}  & \textrm{in the \WP\ case}.
\end{cases}
$$
The aim of the present work is to study the distribution of \textit{real zeroes} of $P_n$ in the above four cases. Explicit formulae for the mean density of real zeroes seem to be available only in the case when the $\xi_k$'s have Gaussian distribution; see~\cite{edelman_kostlan} for an elegant approach based on integral geometry. For example, in the Gaussian case, the mean density of real zeroes of the spherical polynomial $P_n$ at $t\in\R$ is \textit{exactly} $\sqrt n\, (\pi (1+t^2))^{-1}$. In the non-Gaussian case, it is natural to ask about the asymptotic behavior as $n\to\infty$.

\subsection{Assumption on the variance}
Before stating the main result we need to introduce notation which allows to treat all four cases in a unified way.
In the rest of the present paper, $\xi_0,\xi_1,\ldots$ are real-valued i.i.d.\ random variables satisfying
\begin{equation}\label{eq:xi_0_assumpt}
\E \xi_0=0,  \quad \E[\xi_0^2]=1.
\end{equation}

In the \SP\ and \WP\ cases, $P_n(z)$ is a random polynomial defined for all $z\in\CN$. To determine the radius of convergence in the remaining two cases, observe that for every $\eps>0$, with probability $1$ we have $\xi_n = O(\e^{\eps n})$ as $n\to\infty$ by the Borel--Cantelli lemma and Markov inequality.
It follows that in the \HAF\ case, $P_n(z)$ is a random analytic function defined on the open unit disk $\bD = \{|z|<1\}$,
while in the \FAF\ case,  $P_n(z)$ is a random analytic function defined on the entire complex plane, with probability $1$.
In the following, we restrict $z$ to the respective domain on which $P_n$ is defined.

Consider the quantity
\begin{equation}\label{eq:defvn}
v_n(z):=\EW{P^2_n(z)}=\sum_{k=0}^\infty f_{n,k}^2 z^{2k}.
\end{equation}
If  $z$ is real, this is the variance of $P_n(z)$.
In the four cases listed above, $v_n$ is explicitly given by
\begin{equation}\label{eq:vn_explicit_fourcases}
v_n(z)=\begin{cases}
(1+z^2)^n &\textrm{in the \SP\ case, } z\in\CN, \\
\EXP{nz^2} & \textrm{in the \FAF\ case, } z\in\CN,\\
(1-z^2)^{-n} &\textrm{in the \HAF\ case, } |z|<1, \\
\sum_{k=0}^n  \frac{ (nz^{2})^k}{k!} = \frac{\e^{nz^2} \Gamma(n+1, nz^2)}{\Gamma(n+1)} & \textrm{in the \WP\ case, } z\in\CN,
\end{cases}
\end{equation}
where $\Gamma(s,x) = \int_x^\infty \eee^{-t} t^{s-1} \D t$ is the incomplete Gamma function.
All four families of random polynomials fulfill a condition that is sufficient for proving almost everything what follows.
Namely, there exists an open, connected set
$\mathcal{D}\subset \CN$ and an
analytic function $p:\mathcal{D}\to\CN$ such that
\begin{equation}\label{eq:annahmevarmodphi}
\lim_{n\to\infty} \frac{v_n(z)}{\e^{np(z)}}=1
\end{equation}
uniformly on every compact subset of $\mathcal D$.
The function $p$ turns out to determine the limiting mean density  of real zeroes
of $P_n$.
We choose $\mathcal{D}$ and $p:\mathcal{D}\to\CN$ as follows:
\begin{equation}\label{eq:p_corresponding_function}
p(z)=\begin{cases}
\log(1+z^2) & \textrm{on }\mathcal{D}=\CN\backslash\{\I t: -1 < t < 1\} \textrm{ in the \SP\ case,}\\
z^2 & \textrm{on }\mathcal{D}=\CN \textrm{ in the \FAF\ case,}\\
-\log(1-z^2)&  \textrm{on }\mathcal{D}=\mathbb{D} \textrm{ in the \HAF\ case,}\\
z^2 & \textrm{on }\mathcal{D}=\mathbb{D} \cap \{z\in\CN \colon |\e^{-z^2} z^2| < 1/\e\} \textrm{ in the \WP\ case.}
\end{cases}
\end{equation}
The fact that~\eqref{eq:annahmevarmodphi} holds with this choice of $p$ is trivial in the \SP, \FAF\ and \HAF\ cases (where~\eqref{eq:annahmevarmodphi} becomes equality), whereas in the \WP\ case it follows from Proposition 3.1 of~\cite{pritsker_varga}. Note that in the \WP\ case, $\mathcal{D}$ contains the interval $(-1,1)$ in which the real zeroes  will be studied.


\subsection{Main result}
The main result of the present paper identifies the limiting mean density function of real zeroes for  the four families
given in \eqref{eq:thefourcases}.
\begin{theorem}\label{theorem:maintheorem}
Let $\xi_0,\xi_1,\ldots$ be i.i.d.\ real-valued random variables satisfying~\eqref{eq:xi_0_assumpt}.
Let $P_n$ be one of the four random analytic functions defined in~\eqref{eq:defrandompolynomial} and~\eqref{eq:thefourcases}, and
let the corresponding $p:\mathcal{D}\to\CN$ be given by~\eqref{eq:p_corresponding_function}.
If  $N_n[a,b]$ denotes the number of real zeroes of $P_n$ in the interval $[a,b]\subset\mathcal{D}\cap (\R\backslash\{0\})$, then
\begin{equation*}
\lim_{n\to\infty}\frac{\E N_n[a,b]}{\sqrt{n}}=\frac{1}{2\pi}\INT{a}{b}
\sqrt{\frac{p'(t)}{t}+p''(t)} \D t.
\end{equation*}
\end{theorem}
\noindent
More specifically, using the explicit expression for $p$ given in~\eqref{eq:p_corresponding_function}, we obtain the following
\begin{corollary} For $[a,b]\subseteq \R\backslash\{0\}$ we have
\begin{align*}
\lim_{n\to\infty}\frac{\E N_n[a,b]}{\sqrt{n}}
=
\begin{cases}
\frac{1}{\pi} \INT{a}{b} \frac 1{1+t^2}\D t & \textrm{in the \SP\ case,} \\
\frac{b-a}{\pi} & \textrm{in the \FAF\ case.}
\end{cases}
\end{align*}
Furthermore, for $[a,b]\subset (-1,1)\backslash\{0\}$ we have
\begin{align*}
\lim_{n\to\infty}\frac{\E N_n[a,b]}{\sqrt{n}}
=
\begin{cases}
\frac{1}{\pi} \INT{a}{b} \frac 1 {1-t^2}\D t & \textrm{in the \HAF\ case,} \\
\frac{b-a}{\pi} & \textrm{in the \WP\ case.}
\end{cases}
\end{align*}
\end{corollary}

\subsection{Method of proof}
In order to study the zeroes of $P_n$ in an interval $[a,b]$, we fix some $\delta\in (0,1/2)$ and divide the interval $[a,b]$ into pieces of length $\delta n^{-1/2}$. In Theorem~\ref{theo:wkonvgauss} we shall prove that on each such piece $[t,t+\delta n^{-1/2}]$, the appropriately rescaled stochastic process $P_n$ converges to certain stationary Gaussian process denoted by $Z_{\gamma(t)}$; see Figure~\ref{fig:weyl}. From this we derive the distributional convergence of the number of zeroes of $P_n$ in an interval of length $\delta n^{-1/2}$ to the number of zeroes of the stationary Gaussian process in an interval of length $\delta$. The most technically challenging part of the proof is to turn the distributional convergence into convergence of the corresponding expectations. To this end, we show in Section~\ref{sec:unif_integr} that the number of zeroes of $P_n$ in an interval of length $\delta n^{-1/2}$ is a uniformly integrable family of random variables.

\subsection{Comments}

Under the assumption of finite $(2+\delta)$-th moment, Theorem~\ref{theorem:maintheorem}  was established by~\citet{tao_vu} who used the replacement method coming from random matrix theory. Our method is completely different. Besides relaxing the moment assumption, our method may have some other advantages. For example, it can be applied in the setting when the $\xi_k$'s belong to the domain of attraction of a stable law, or to study real zeroes of other random analytic functions.  Similar approach has already been applied to real zeroes of random trigonometric polynomials in one and many variables~\cite{flasche,angst_poly_viet}, as well as to random Taylor series~\cite{flaschekabluchko2017}. It should be possible to treat random linear combinations of orthogonal polynomials by similar methods, thus proving the universality of the limiting mean density function computed in~\cite{lubinsky1,lubinsky2} in the case of Gaussian coefficients.

We believe that the second moment assumption on $\xi_0$ is nearly optimal. In fact, it could be replaced by the slightly more general assumption that $\xi_0$ belongs to the domain of attraction of the normal law. The proofs would remain almost the same except that a slowly varying factor would appear. To keep the notation simple, we decided to restrict ourselves to the case when $\E [\xi_0^2]<\infty$. On the other hand, we conjecture that coefficients from a stable domain of attraction should lead to a different limiting mean density function.
Presumably, Theorem~\ref{theorem:maintheorem} holds even if the interval $[a,b]$ contains $0$, but our method does not allow to prove this because the proof of Lemma~\ref{lemma:abschaetzungdmeins} breaks down at $t=0$.

\begin{figure}[th]
\includegraphics[height=0.3\textwidth, width=0.99\textwidth]{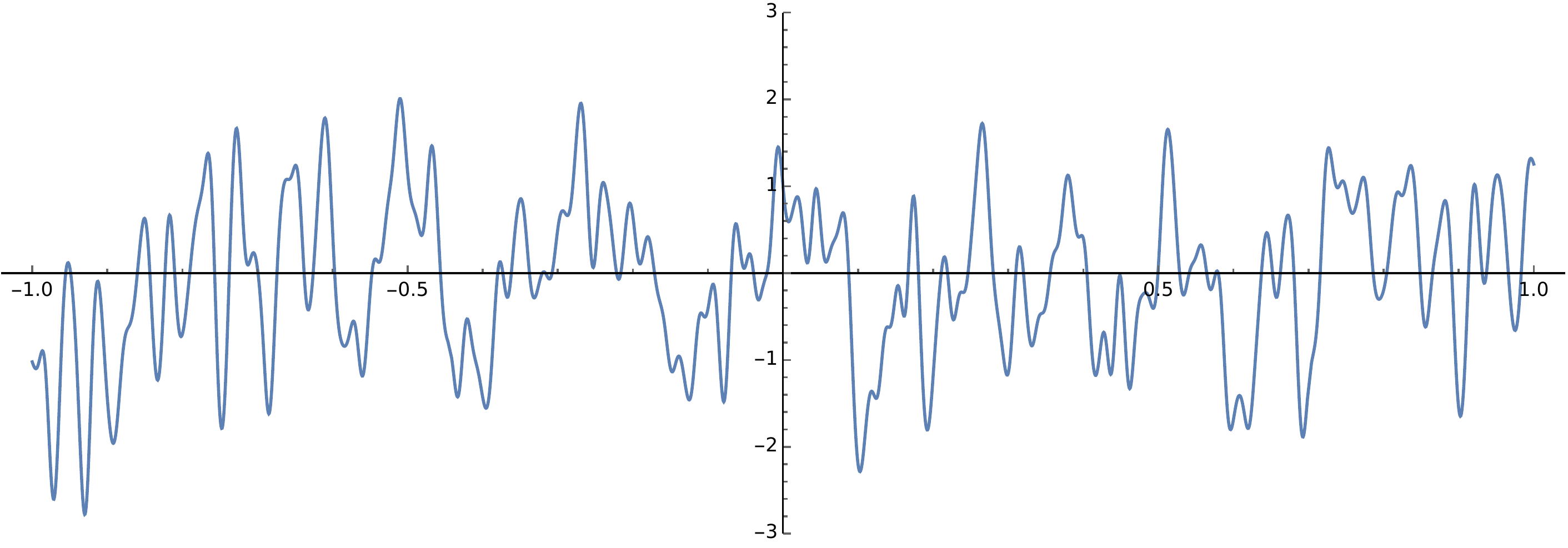}
\caption
{
The normalized Weyl polynomial $P_n(z)/ \sqrt{v_n(z)}$ for $z\in (-1,1)$. Here, the degree is $n=10^4$ and $\P{\xi_0=\pm 1}=1/2$. In Theorem~\ref{theo:wkonvgauss}, this process will be approximated (in windows of size $\text{const} \cdot n^{-1/2}$) by a stationary Gaussian process.}\label{fig:weyl}
\end{figure}

\subsection{Notation}
In the following, $C>0$ (respectively, $c>0$) denotes a sufficiently large  (respectively, small) constant that does not depend on $n$ and may change from line to line. Usually, our statements hold for sufficiently large $n\geq n_0$ only, where the number $n_0$ also changes from line to line.   The floor and the ceiling functions of $x$ are denoted by $\lfloor x \rfloor$ and $\lceil x \rceil$, respectively.

\section{Weak convergence to a stationary Gaussian process in a small window}
\noindent In this section we prove that in a suitable metric space, the appropriately rescaled sequence $(P_n)_{n\in\N}$ converges
weakly  to a stationary Gaussian process.

\subsection{Metric space of analytic functions}
The open and the closed disks of radius $R>0$ will be denoted by
$$
\bD_R=\{z\in\CN\colon |z|< R\}
\quad\text{and}\quad
\overline{\bD}_R=\{z\in\CN\colon |z|\leq R\},
$$
respectively.  Let $A(\bD_R)$ be the space of all functions $f: \overline{\bD}_R \to \CN$ that are continuous on $\overline{\bD}_R$ and analytic on the open disk $\bD_R$. Endowed with the uniform norm, $A(\bD_R)$ becomes a Banach space and even a Banach algebra (the disk algebra).
Let $A_{\text{real}}(\bD_R)$ be the closed subset of $A(\bD_R)$ consisting of those functions which take real values on $\overline{\bD}_R \cap \R$.

\subsection{The limit process}
For $\gamma >0$ let $(Z_{\gamma}(t))_{t\in\R}$ be a stationary, real-valued  Gaussian process with zero mean and covariance function
\begin{equation}\label{eq:covariancelimitfunction}
\COV\left[{Z_{\gamma}(s),Z_{\gamma}(t)}\right]=\EXP{- \frac{\gamma}{2}(t-s)^2}, \quad t,s\in\R.
\end{equation}
This process can be extended to a random analytic function on the entire complex plane. The probably simplest way to do this is to  define
\begin{equation}\label{eq:limitfunction}
Z_{\gamma}(u):= \eee^{-\gamma u^2/2} \sum_{k=0}^\infty \zeta_k \frac{\gamma^{k/2}u^k}{\sqrt{k!}},
\quad u\in\CN,
\end{equation}
where $\zeta_1,\zeta_2,\ldots$ are i.i.d.\ real standard Gaussian random variables. Using that $|\zeta_k| = O(\sqrt{\log k})$, $k\to\infty$, a.s., it is easy to check that the right-hand side of~\eqref{eq:limitfunction} converges uniformly on compact sets and hence defines an analytic function of $u$, with probability $1$. Then, one easily checks that~\eqref{eq:covariancelimitfunction} holds. Moreover, for all $t,s\in\CN$ we have
\begin{equation}\label{eq:cov_Z_gamma}
\E\left[{Z_{\gamma}(s)Z_{\gamma}(t)}\right]=\eee^{-\frac{\gamma}{2}(t-s)^2},
\quad
\E\left[{Z_{\gamma}(s)\overline{Z_{\gamma}(t)}}\right]=\eee^{-\frac{\gamma}{2}(\overline{t}-s)^2}.
\end{equation}
The expected number of real zeroes of this stationary Gaussian process is recorded in the next lemma.
\begin{lemma}\label{lemma:exprootsofstatgaussianprocessunitinterval}
Let $N^{(\gamma)}_{\infty}[a,b]$ be the number of real zeroes of $Z_{\gamma}$ in the interval $[a,b]\subset \R$.  Then,
\begin{equation*}
\E N_\infty^{(\gamma)}[a,b] = \frac{(b-a)}{\pi} \sqrt{\gamma}.
\end{equation*}
\end{lemma}
\begin{proof}
The covariance function of $Z_{\gamma}$ satisfies
$$
\COV\left[{Z_{\gamma}(0),Z_{\gamma}(t)}\right]=\eee^{-\frac 12 \gamma t^2} = 1- \frac12 \gamma t^2 + o(t^2) \quad \text{ as } t\to 0,
$$
and the claim follows from the Rice formula; see for example \cite[Chapter 10]{cramer2013stationary} or~\cite[Eq.~(3.2) on p.~71]{azais_wschebor_book}).
\end{proof}

\subsection{Weak convergence of the random analytic function}
Let $P_n$ be a random analytic function from one of the four ensembles introduced in Section~\ref{subsec:four_families}. Fix $\delta >0$. We shall show that in a small window of size $\delta n^{-1/2}$ around some point $t\in \mathcal{D}\cap (\R\backslash\{0\})$ the process $P_n$ looks, upon a proper rescaling and as $n\to\infty$, like the stationary Gaussian process introduced in~\eqref{eq:limitfunction}.

Take some $t\in \mathcal{D}\cap (\R\backslash\{0\})$ and consider the process $Q_{n,t}$ given by
\begin{equation}\label{eq:definitionQ}
Q_{n,t}(z):=\left(v_n\left(t+\frac{z}{\sqrt{n}}\right)\right)^{-1/2} P_n\left(t+\frac{z}{\sqrt{n}}\right).
\end{equation}
If we fix some radius $R>0$, then for all sufficiently large $n$, $Q_{n,t}(z)$  is a well-defined random analytic function on $\overline{\bD}_R$ and we may consider it as a random element of $A_{\rm{real}}(\bD_R)$.
Let us now state the functional limit theorem.
\begin{theorem}\label{theo:wkonvgauss}
Fix some $t\in \mathcal{D}\cap (\R\backslash \{0\})$.
Then, for all $R>0$,
\begin{equation*}
Q_{n,t} \toweak Z_{\gamma(t)} \quad \textrm{on } A_{\rm{real}}(\bD_R),
\end{equation*}
where $\toweaknon$ denotes weak convergence and $\gamma(t)>0$ is explicitly given by
\begin{equation}\label{eq:gamma}
\gamma(t)=\frac{1}{4}\left(\frac{p'(t)}{t}+p''(t)\right).
\end{equation}
\end{theorem}
\begin{remark}
The theorem breaks down if $t=0$. Indeed, we have $P_n(0) = f_{n,0} \xi_0$, so there can be no central limit theorem if $\xi_0$ is not Gaussian. For other random analytic functions, functional limit theorems of the above type appeared in~\cite{ledoan,bleher_di,iksanov2016,flaschekabluchko2017}.
\end{remark}

\begin{proof}[Proof of Theorem~\ref{theo:wkonvgauss}]
We can write $Q_{n,t}(z) = \sum_{k=0}^{\infty} a_{n,k}(z)\xi_k$, where $a_{n,0}(z), a_{n,1}(z),\ldots$ are analytic functions given by
\begin{equation}\label{eq:a_n_k_def}
a_{n,k}(z):=\frac{f_{n,k}
\left(t+\frac{z}{\sqrt{n}}\right)^k}{\sqrt{v_n(t+\frac{z}{\sqrt{n}})}}, \quad |z|\leq R.
\end{equation}

\vspace*{2mm}
\noindent
\textit{Convergence of the finite-dimensional distributions.} We need to show that for every $d\in\N$ and every $z_1,\ldots, z_d\in \overline{\bD}_R$,
\begin{equation}\label{eq:conv_finit_distr}
\begin{pmatrix}
\RE Q_{n,t}(z_1)\\
\IM Q_{n,t}(z_1) \\
\vdots \\
\RE Q_{n,t}(z_d) \\
\IM Q_{n,t}(z_d)
\end{pmatrix}
\todistr
\begin{pmatrix}
\RE Z_{\gamma(t)}(z_1)\\
\IM Z_{\gamma(t)}(z_1)\\
\vdots \\
\RE Z_{\gamma(t)}(z_d) \\
\IM Z_{\gamma(t)}(z_d)
\end{pmatrix}.
\end{equation}
The left-hand side can be represented as $\sum_{k=0}^{\infty}V_{n,k}$, where $V_{n,k}$ is the $\R^{2d}$-valued random vector defined by
\begin{equation*}
V_{n,k}:=
\begin{pmatrix}
	V_{n,k}(1) \\
	V_{n,k}(2) \\
	\vdots \\
	V_{n,k}(2d-1) \\
	V_{n,k}(2d)
	\end{pmatrix}
:=
\begin{pmatrix}
\RE a_{n,k}(z_1)\\
\IM a_{n,k}(z_1) \\
\vdots \\
\RE a_{n,k}(z_d) \\
\IM a_{n,k}(z_d)
\end{pmatrix}\xi_k.
\end{equation*}
We apply the multivariate version of the Lindeberg CLT stated in Proposition~6.2 of \cite{flaschekabluchko2017} to a triangular array whose $n$-th row consists of $V_{n,0}, V_{n,1},\ldots$.
To prove the convergence of the covariances it suffices to show that for all $z_i,z_j\in \overline{\bD}_R$,
\begin{align}\label{eq:cov_one}
\lim_{n\to\infty}\EW{Q_{n,t}(z_i) Q_{n,t}(z_j)}&=\EW{Z_{\gamma(t)}(z_i)
Z_{\gamma(t)}(z_j)},\\
\label{eq:cov_two}
\lim_{n\to\infty}\EW{Q_{n,t}(z_i) \overline{Q_{n,t}(z_j)}}&=\EW{Z_{\gamma(t)}(z_i)\overline{Z_{\gamma(t)}(z_j)}},
\end{align}
because the  covariance matrix of $\sum_{k=0}^{\infty} V_{n,k}$
can be expressed linearly in terms of \eqref{eq:cov_one} and \eqref{eq:cov_two}. For the first expectation
we have
\begin{align}
\EW{Q_{n,t}(z_i)Q_{n,t}(z_j)}
&=
\frac{\sum_{k=0}^\infty f_{n,k}^2
\left(t+\frac{z_i}{\sqrt{n}}\right)^k\left(t+\frac{z_j}{\sqrt{n}}\right)^k}
{\sqrt{v_n\left(t+\frac{z_i}{\sqrt{n}}\right)v_n\left(t+\frac{z_j}{\sqrt{n}}\right)}} \label{eq:E_Q_n_t_Q_n_t}\\
&=
\frac{
v_n\left(
\sqrt{t^2+t\frac{z_i+z_j}{\sqrt{n}}+\frac{z_iz_j}{n}}
\right)
}{
\sqrt{v_n\left(t+\frac{z_i}{\sqrt{n}}\right)v_n\left(t+\frac{z_j}{\sqrt{n}}\right)}
}
\sim \frac{\EXP{np\left(\sqrt{t^2+t\frac{z_i+z_j}{\sqrt{n}}+\frac{z_iz_j}{n}}\right)}}
{\EXP{\frac{n}{2}\left(p\left(t+\frac{z_i}{\sqrt{n}}\right)
+p\left(t+\frac{z_j}{\sqrt{n}}\right)\right)}} \notag
\end{align}
as $n\to\infty$, where we utilized~\eqref{eq:annahmevarmodphi} in the last step.
Furthermore, we have
\begin{equation*}
\sqrt{t^2+t\frac{z_i+z_j}{\sqrt{n}}+\frac{z_iz_j}{n}}
=
t+\frac{z_i+z_j}{2\sqrt{n}}-\frac{(z_i-z_j)^2}{8nt}+o\left(\frac{1}{n}\right),
\end{equation*}
as $n\to\infty$,
and since $p$ is supposed to be analytic,
\begin{equation}\label{eq:lemma4_asymptoticzwei}
p\left(\sqrt{t^2+t\frac{z_i+z_j}{\sqrt{n}}+\frac{z_iz_j}{n}}\right)
=
p(t)
+
p'(t)\left(\frac{z_i+z_j}{2\sqrt{n}}-\frac{(z_i-z_j)^2}{8nt}\right)
\\
+\frac{p''(t)}{2}\frac{(z_i+z_j)^2}{4n}+o\left(\frac{1}{n}\right).
\end{equation}
In addition, we have for $l=i,j$,
\begin{equation}\label{eq:lemma4_asymptoticdrei}
p\left(t+\frac{z_l}{\sqrt{n}}\right)
=p(t)+p'(t)\frac{z_l}{\sqrt{n}}+\frac{p''(t)}{2}\frac{z_l^2}{n}+
o\left(\frac{1}{n}\right).
\end{equation}
Taking~\eqref{eq:lemma4_asymptoticzwei} and~\eqref{eq:lemma4_asymptoticdrei} into account and using the identity $\overline{Q_{n,t}(z)}=Q_{n,t}(\overline{z})$,
we arrive at
\begin{align}
\label{eq:covariancestructureeins}
\lim_{n\to\infty}
\EW{Q_{n,t}(z_i)Q_{n,t}(z_j)}&=\EXP{-\frac 12\gamma(t)(z_i-z_j)^2}, \\
\label{eq:covariancestructurezwei}
\lim_{n\to\infty}
\EW{Q_{n,t}(z_i)\overline{Q_{n,t}(z_j)}}&=\EXP{-\frac 12 \gamma(t)(z_i-\overline{z}_j)^2},
\end{align}
where
\begin{equation}\label{eq:definitiongamma}
\gamma(t):=\frac{1}{4}\left(\frac{p'(t)}{t}+p''(t)\right).
\end{equation}
This proves~\eqref{eq:cov_one} and~\eqref{eq:cov_two} in view of the formula for the covariance function of $Z_{\gamma}$ given in~\eqref{eq:cov_Z_gamma}.

It remains to verify the Lindeberg condition, namely
\begin{equation}\label{eq:lindeberg}
\lim_{n\to\infty}\sum_{k=0}^\infty \EW{V_{n,k}^2(i)\IND{\{|V_{n,k}(i)|>\eps\}}}=0
\end{equation}
for all $i=1,\ldots,2d$ and $\eps>0$.
Define
\begin{equation*}
b_{n,k}:=\max_{i=1,\dots, d} |a_{n,k}(z_i)| \quad \textrm{and} \quad
\tilde{b}_n:=\max_{k=0,1,\ldots} b_{n,k}.
\end{equation*}
Then, $V_{n,k}^2(i)\leq b_{n,k}^2\xi_k^2$ for all $i=1,\ldots,2d$ and thus, for every $\eps>0$,
\begin{align*}
\sum_{k=0}^\infty \EW{V^2_{n,k}(i)\IND{\{|V_{n,k}(i)|>\eps\}}}
& \leq \sum_{k=0}^\infty b_{n,k}^2 \EW{\xi_k^2\IND{\{|\xi_k|>\eps/b_{n,k}\}}} \\
& \leq \EW{\xi_1^2\IND{\{|\xi_1|>\eps / \tilde{b}_{n}\}}}\sum_{k=0}^\infty b_{n,k}^2 \\
& \leq \EW{\xi_1^2\IND{\{|\xi_1|>\eps /\tilde{b}_{n}\}}}\sum_{k=0}^\infty \sum_{j=1}^d |a_{n,k}(z_j)|^2 \\
& \leq C \, \EW{\xi_1^2\IND{\{|\xi_1|>\eps / \tilde{b}_{n}\}}}.
\end{align*}
In the last step we used that for all $j=1,\dots,d$,
$$
\sum_{k=0}^\infty |a_{n,k}(z_j)|^2
=
\E [Q_{n,t}(z_j) \overline{Q_{n,t}(z_j)}]
\ton \EXP{2 \gamma(t)(\IM z_j)^2}, \quad \text{ as } n\to\infty,
$$
by~\eqref{eq:covariancestructurezwei}.
Finally,
\begin{equation*}
\lim_{n\to\infty} \EW{\xi_1^2\IND{\{|\xi_1|>\eps /\tilde{b}_{n}\}}}=0
\end{equation*}
because $\EW{\xi_1^2}=1<\infty$  and  $\lim_{n\to\infty}\tilde{b}_n=0$  by Lemma~\ref{lemma:locallimittheorem} whose statement and proof are postponed to Section~\ref{sec:auxiliary}.
This verifies the Lindeberg condition~\eqref{eq:lindeberg} and completes the proof of~\eqref{eq:conv_finit_distr}.

\vspace*{2mm}
\noindent\textit{Tightness.} To complete the proof of the theorem we need to show that the probability laws of $(Q_{n,t})_{n\in\N}$ form a tight sequence on $A_{\rm{real}}(\bD_R)$.
For random analytic functions, there are especially simple criteria of tightness.
Namely, by~\cite[Remark on p.~341]{shirai}, it suffices to show that $\E|Q_{n,t}(z)|^2 \leq  C$ for all $z\in \overline{\mathbb{D}}_R$ and all sufficiently large $n\in \N$. But~\eqref{eq:covariancestructurezwei} (which holds uniformly over $z_i, z_j\in \overline{\bD}_R$) yields
\begin{align} \label{eq:qvarabsch}
\E \left|Q_{n,t}(z)\right|^2=
\EW{Q_{n,t}(z)\overline{Q_{n,t}(z)}}
=\EXP{2\gamma(t)(\IM z)^2+o(1)}
<C,
\end{align}
thus completing the proof of Theorem~\ref{theo:wkonvgauss}.
\end{proof}

\subsection{Distributional convergence of the number of zeroes}
The next lemma transfers the weak convergence of the scaled random analytic functions
$Q_{n,t}$ on $A_{\rm{real}}(\bD_R)$  to the convergence in distribution of the corresponding random number of zeroes
in small windows of length $\delta/\sqrt{n}$.
\begin{lemma}\label{lemma:hauptschwachekonvergenz}
Fix some $\delta>0$ and $t\in \mathcal{D}\cap (\R\backslash \{0\})$. Then, the sequence of random variables
\begin{equation}\label{eq:seq_N_n}
\left(N_n \left[t, t+\frac{\delta}{\sqrt n}\right]\right)_{n\in\N}
\end{equation}
converges in distribution to the number of real zeroes of the Gaussian process $Z_{\gamma(t)}(\cdot)$ in the interval $[0,\delta]$.
\end{lemma}
\begin{proof}
By~\eqref{eq:definitionQ}, $N_n [t, t+ \delta/\sqrt n]$ is the number of real zeroes of $Q_{n,t}(\cdot)$ in the interval $[0,\delta]$. By Theorem~\ref{theo:wkonvgauss}, the latter process converges weakly to $Z_{\gamma(t)}(\cdot)$ on the space $A_{\rm{real}}(\bD_R)$, as $n\to\infty$.
We may take $R>\delta$, so that the interval $[0,\delta]$ is contained in the interior of the disk $\mathbb{D}_R$ of radius $R$.
To pass to the distributional convergence of real zeroes, we employ the continuous mapping theorem in the same way as it was done in~\cite{iksanov2016}. By Lemma~4.1 therein, the map which assigns to a function $f\in A_{\rm{real}}(\bD_R) \backslash\{0\}$ the  number of zeroes of $f$ in the interval $[0,\delta]$  is locally constant (hence, continuous) on the set of all analytic functions which do not vanish at $0,\delta$ and have no multiple zeroes in the interval $[0,\delta]$. This set has full measure w.r.t.\ the law of $Z_{\gamma(t)}$ (the a.s.\ absence of multiple zeroes follows from the Bulinskaya lemma; see, e.g.,~\cite[Lemma~4.3]{iksanov2016}). Hence, the continuous mapping theorem implies the distributional convergence of  $N_n [t, t+\delta/\sqrt n]$ to the number of zeroes of $Z_{\gamma(t)}$ in $[0,\delta]$ as $n\to\infty$.
\end{proof}

\subsection{Proof of Theorem~\ref{theorem:maintheorem} assuming uniform integrability}
As we shall prove in Lemma~\ref{lemma:hauptgnbeschraenkt}, below, the sequence~\eqref{eq:seq_N_n} is uniformly integrable for $0<\delta<1/2$. Assuming this, we can prove Theorem~\ref{theorem:maintheorem} as follows.
Let $\eps_0>0$ be so small that the interval $[a-2\eps_0, b+2\eps_0]$ is contained in $\mathcal D \cap (\R \backslash\{0\})$.  Define a  function $g_n:[a-\eps_0,b+\eps_0]\to [0,\infty)$ as follows:
$$
g_n(t) = \E N_n\left[t, t + \frac{\delta}{\sqrt n}\right].
$$
It follows from this definition that for sufficiently large $n$,
\begin{equation}\label{eq:sandwich}
\int_a^{b - (\delta/\sqrt n)} g_n (t) \D t  \leq  \frac{\delta}{\sqrt n} \E N_n[a,b] \leq \int_{a-(\delta/\sqrt n)}^{b} g_n (t) \D t.
\end{equation}
The distributional convergence stated in Lemma~\ref{lemma:hauptschwachekonvergenz} and the uniform integrability implied by Lemma~\ref{lemma:hauptgnbeschraenkt} yield the convergence of the expectations:
\begin{equation}\label{eq:E_N_n_local_conv}
\lim_{n\to\infty} g_n(t) = \lim_{n\to\infty} \E N_n \left[t, t+\frac{\delta}{\sqrt n}\right]  = \E  N_\infty^{(\gamma(t))}[0,\delta] = \frac{\delta}{\pi} \sqrt{\gamma(t)}
=
\frac{\delta}{2\pi}\sqrt{\frac{p'(t)}{t}+p''(t)}
\end{equation}
for all $t\in [a-\eps_0, b+\eps_0]$; see also Lemma~\ref{lemma:exprootsofstatgaussianprocessunitinterval} for the expected number of zeroes of the limit process.
Also, it follows from Lemma~\ref{lemma:hauptgnbeschraenkt} that $0\leq g_n(t) \leq C$ for some constant $C>0$ and all sufficiently large $n$. Utilizing the dominated convergence theorem, we arrive at
$$
\lim_{n\to\infty} \int_a^{b - (\delta/\sqrt n)} g_n (t) \D t = \lim_{n\to\infty} \int_{a-(\delta/\sqrt n)}^{b} g_n (t) \D t = \frac{\delta}{2\pi}\INT{a}{b}
\sqrt{\frac{p'(t)}{t}+p''(t)} \D t.
$$
The sandwich lemma, applied to~\eqref{eq:sandwich}, yields
$$
\lim_{n\to\infty} \frac{\E N_n[a,b]}{\sqrt n} = \frac{1}{2\pi}\INT{a}{b}
\sqrt{\frac{p'(t)}{t}+p''(t)} \D t,
$$
thus completing the proof of Theorem~\ref{theorem:maintheorem}.

\section{Uniform integrability of \texorpdfstring{$N_n$}{Nn} on intervals of length \texorpdfstring{$\delta n^{-1/2}$}{delta sqrt n}}\label{sec:unif_integr}

\subsection{Statement of the main lemmas}
We recall that $N_n[a,b]$ is the number of real zeroes of $P_n$ (defined by~\eqref{eq:defrandompolynomial} and~\eqref{eq:thefourcases}) in the interval $[a,b]\subset \mathcal{D}\cap \R$. We aim to prove the following

\begin{lemma}\label{lemma:hauptgnbeschraenkt}
Fix some  $0<\delta<1/2$ and a compact set $K\subset \mathcal{D}\cap (\R\backslash \{0\})$. Let also $1<\kappa<2$. Then there exists a constant $C>0$  such that for all sufficiently large $n\in\N$ and  all $t\in K$,
\begin{equation*}
\EW{N^\kappa_n\l[t,t+\frac{\delta}{\sqrt{n}}\r]}<C.
\end{equation*}
\end{lemma}

\noindent
We shall deduce Lemma~\ref{lemma:hauptgnbeschraenkt} from the following two statements whose proofs are postponed:

\begin{lemma}\label{lemma:manyzeroesareunlikely}
Fix $1<\kappa<2$, an interval $[a,b]\subset \mathcal{D}\cap \R$ and $\epsilon>0$.  Then,
\begin{equation*}
\lim_{n\to\infty} \EW{N^\kappa_n[a,b]\IND{\{N_n[a,b]\geq \e^{n^\eps}\}}}=0.
\end{equation*}
\end{lemma}

\begin{lemma}\label{lemma:abschaetzungdm}
Fix some sufficiently small $0<\delta<1/2$ and a compact set $K\subset \mathcal{D}\cap (\R\backslash\{0\})$.
Then there exist constants $C>0$ and $c>0$ such that
\begin{equation*}
\P{N_n[t,t+\delta n^{-1/2}] \geq m} \leq C\left(\left(2\delta\right)^{(2/3) m}+\left(2\delta\right)^{-(1/3)m}\EXP{-c n^{1/5}}\right)
\end{equation*}
for all $n\geq n_0$, $m\in\N$,  and $t\in K$.
\end{lemma}

\begin{proof}[Proof of Lemma~\ref{lemma:hauptgnbeschraenkt} given Lemmas~\ref{lemma:manyzeroesareunlikely} and~\ref{lemma:abschaetzungdm}]
For $t\in K$ we write
\begin{equation*}
N_{n,t,\delta}:=N_n\left[t,t+\frac{\delta}{\sqrt{n}}\right].
\end{equation*}
We can find an interval $[a,b] \subset \mathcal{D}\cap \R$ such that $K\subset (a,b)$. Take some $\eps>0$.  Then, for sufficiently large $n$, we have $N_{n,t,\delta} \leq  N_n[a,b]$ and hence
$$
\EW{N^\kappa_{n,t,\delta}\IND{\{N_{n,t,\delta}\geq\e^{n^\eps}\}}} \leq \EW{N^\kappa_n[a,b]\IND{\{N_n[a,b]\geq \e^{n^\eps}\}}}<C
$$
in view of Lemma \ref{lemma:manyzeroesareunlikely}. Thus, it suffices to show that
$\EW{N^\kappa_{n,t,\delta}\IND{\{N_{n,t,\delta}\leq\e^{n^\eps}\}}}<C$.
Let $m_0=m_0(n,\delta)$ be chosen such that
\begin{equation}\label{eq:festlegung_mnull}
\EXP{\frac{c}{4}n^{1/5}}\leq (2\delta)^{-m_0/3} \leq \EXP{\frac{c}{2}n^{1/5}},
\end{equation}
where $c>0$ is the small constant from Lemma~\ref{lemma:abschaetzungdm}.
It follows from \eqref{eq:festlegung_mnull} that
\begin{equation}\label{eq:festlegung_mnull1}
(2\delta)^{2m_0/3} \leq \EXP{-\frac{c}{2}n^{1/5}}.
\end{equation}
Observe that $\{N_n[t,t+\delta n^{-1/2}] \geq 1\}\supseteq
\{N_n[t,t+\delta n^{-1/2}] \geq 2\} \supseteq \ldots$ is a non-increasing sequence of events. Therefore,
\begin{multline*}
\EW{N^\kappa_{n,t,\delta}\IND{\{N_{n,t,\delta}\leq \e^{n^\eps}\}}}\leq
\EW{N^2_{n,t,\delta}\IND{\{N_{n,t,\delta}\leq \e^{n^\eps}\}}}
\leq \sum_{m=1}^{\lfloor \e^{n^\eps} \rfloor}
(2m-1)\P{N_n[t,t+\delta n^{-1/2}] \geq m} \\
\leq \sum_{m=1}^{m_0} (2m-1)\P{N_n[t,t+\delta n^{-1/2}] \geq m}
+\P{N_n[t,t+\delta n^{-1/2}] \geq m_0}\sum_{m=m_0+1}^{\lfloor \e^{n^\eps} \rfloor} (2m-1),
\end{multline*}
where we observe that $m_0< \eee^{n^\eps}$ for sufficiently large $n$, as follows from~\eqref{eq:festlegung_mnull}.
Applying  Lemma~\ref{lemma:abschaetzungdm} we obtain, for sufficiently small $\eps>0$,
\begin{align*}
\EW{N^\kappa_{n,t,\delta}\IND{\{N_{n,t,\delta}\leq \e^{n^\eps}\}}}
&\leq
C\sum_{m=1}^\infty (2m-1) \left(2\delta\right)^{2m/3}
+C\EXP{-c n^{1/5}}\sum_{m=1}^{\lfloor\e^{n^\eps}\rfloor} (2m-1) \left(2\delta\right)^{-1/3m} \\
&\quad  +C\e^{2n^\eps} \left((2\delta)^{2m_0/3}+(2\delta)^{-m_0/3}\EXP{-cn^{1/5}}\right)  \\
&\leq C+C\e^{2n^\eps}\EXP{-\frac{c}{2}n^{1/5}}<C,
\end{align*}
where we used~\eqref{eq:festlegung_mnull} and~\eqref{eq:festlegung_mnull1}.
\end{proof}

\subsection{Proof of Lemma~\ref{lemma:manyzeroesareunlikely}}
We fix $1<\kappa<2$, an interval $[a,b]\subset \mathcal{D}\cap \R$ and $\epsilon>0$. Our aim is to prove that
\begin{equation}\label{eq:N_kappa_ab_leq_C}
\lim_{n\to\infty} \EW{N^\kappa_n[a,b]\IND{\{N_n[a,b]\geq \e^{n^\eps}\}}}=0.
\end{equation}
\begin{proof}[Proof of~\eqref{eq:N_kappa_ab_leq_C}]
First of all, the statement is trivial for spherical and Weyl polynomials because the number of real zeroes of a degree $n$ polynomial is bounded by $n$. In the following, we consider the \FAF\ and the \HAF\ cases only. In particular, the coefficients $f_{n,k}$ do not vanish.
The first step of the proof uses an argument based on the Jensen theorem which follows an idea of~\cite{ibragimov_maslova1} as developed in~\cite{flaschekabluchko2017}.  Since $\P{\xi_0=0} \neq 1$, we can choose a
sufficiently small $0<\eta<1$ such that
\begin{equation*}
q:=\P{|\xi_0|\leq \e\eta}<1.
\end{equation*}
For $k=0,1,\dots $ consider the events
\begin{equation*}
B_k:=\left\{|\xi_0|\leq \e \eta,\dots,|\xi_{k-1}|\leq \e \eta, |\xi_k|>\e \eta\right\}.
\end{equation*}
Keep in mind that
\begin{equation*}
\quad \P{B_k}=q^{k}(1-q) \quad \textrm{and}\quad \bigcup_{k=0}^\infty B_k=\Omega \mod \mathbb P,
\end{equation*}
where $(\Omega, \mathcal A, \mathbb P)$ is the probability space we are working on.
Let $P_n^{(k)}$ denote the $k$-th derivative of $P_n$. On the event $B_k$ we have
\begin{equation*}
\left|P_n^{(k)}(0)\right|=k! f_{n,k} |\xi_k| \geq k! f_{n,k} \e \eta \geq k! f_{n,k} \eta.
\end{equation*}
We shall use the abbreviation $N_n=N_n[a,b]$. By the theorem of Rolle, $N_n$ can be upperbounded by the number of zeroes of $P_n^{(k)}$ in the interval $[a,b]$ plus $k$. Choosing $r:=\max\{|a|,|b|\}$, we can estimate $N_n$ by the number of zeroes of $P_n^{(k)}$ in the disk $\bar \bD_r$ plus $k$.  Jensen's theorem (see, e.g.~\cite[pp.~280--281]{conway_book}) applied to $P_n^{(k)}$ yields that on the event $B_k$,
\begin{equation}
\label{eq:estimateNnBk}
N_n\leq k+ \frac 1 {\log\l(R/r\r)}\log\l(\frac{\sup_{|z|=R}{|P_n^{(k)}(z)|}}{|P_n^{(k)}(0)|}\r) \\
\leq k+C\log\l(\frac{1}{\eta}\sum_{j=k}^\infty \binom{j}{k} \frac{f_{n,j}}{f_{n,k}} |\xi_j| R^{j-k}\r),
\end{equation}
where we have chosen $R$ such that $r<R\in \mathcal{D}\cap \R$.
Besides, on $B_k$ we have
\begin{equation}\label{eq:logquadrat_concave}
\frac{1}{\eta}\sum_{j=k}^\infty \binom{j}{k} \frac{f_{n,j}}{f_{n,k}} |\xi_j| R^{j-k}
\geq \frac{|\xi_k|}{\eta}\geq  \e.
\end{equation}
On the other hand, we shall show in Lemma~\ref{lem:est_det_poly}, below, that there is $C>0$ such that for all $n\in\N$, $k\in\N_0$,
\begin{equation}\label{eq:uniform_abel}
\sum_{j=k}^{\infty} \binom{j}{k} \frac{f_{n,j}}{f_{n,k}} R^{j-k}
\leq \frac{C}{f_{n,k}}\left(\frac{2}{R\log(A/R)}\right)^{k}
\EXP{\frac{n}{2}p\left(\sqrt{AR}\right)},
\end{equation}
where $R<A\in \mathcal{D}\cap \R$ is a constant independent of $n$ and $k$ and
$p:\mathcal{D}\to\CN$ is given by \eqref{eq:p_corresponding_function}.
Using the same idea as in the standard proof of the Markov inequality, we obtain the estimate
\begin{align*}
\EW{N_n^\kappa\IND{\l\{N_n\geq  \e^{n^\eps}\r\}\cap B_k}}
&=
\EW{\frac{N_n^2}{N_n^{2-\kappa}}\IND{\l\{N_n\geq  \e^{n^\eps}\r\}\cap B_k}}
\leq
\EW{\frac{N_n^2}{\e^{n^\eps (2-\kappa)}}\IND{\l\{N_n \geq  \e^{n^\eps}\r\}\cap B_k}} \\
&\leq \e^{-n^\eps(2-\kappa)} \EW{N^2_n\IND{B_k}}
\end{align*}
which holds for all $k=0,1,\ldots$. The inequality
$(a+b)^2\leq 2a^2+2b^2$,
for all $a,b\in\R$, combined with~\eqref{eq:estimateNnBk} allows us to conclude that
\begin{align}\label{eq:est_one}
\EW{N_n^\kappa\IND{\l\{N_n \geq \e^{n^\eps}\r\}\cap B_k}}
\leq C\e^{-n^\eps(2-\kappa)}\l(k^2\P{B_k} +\EW{\log^2\l(\frac{1}{\eta}\sum_{j=k}^\infty \binom{j}{k} |\xi_j| \frac{f_{n,j}}{f_{n,k}} R^{j-k}\r)\IND{B_k}}\r).
\end{align}
Since the function $x\mapsto\log^2(x)$ is concave for $x\geq \e$ and in view of~\eqref{eq:logquadrat_concave}, we may use the inequality of Jensen (on the event $B_k$) to obtain the estimate
\begin{align*}
\EW{\log^2\l(\frac{1}{\eta}\sum_{j=k}^\infty \binom{j}{k} |\xi_j| \frac{f_{n,j}}{f_{n,k}} R^{j-k}\r)\IND{B_k}}
\leq \P{B_k} \log^2\l(\frac{1}{\P{B_k}}\EW{\frac{1}{\eta}\sum_{j=k}^\infty \binom{j}{k} |\xi_j| \frac{f_{n,j}}{f_{n,k}} R^{j-k}\IND{B_k}}\r).
\end{align*}
Treating the term with $j=k$ separately, using the independence of $(\xi_k)_{k\in\N_0}$, the observation $\EW{|\xi_j|\IND{B_k}} = \P{B_k} \E |\xi_1| \leq \P{B_k}$ for $j>k$, and~\eqref{eq:uniform_abel},  we obtain for sufficiently large $n$ the estimate
\begin{align}\label{eq:est_two}
\begin{split}
&\EW{\log^2\l(\frac{1}{\eta}\sum_{j=k}^\infty \binom{j}{k} |\xi_j| \frac{f_{n,j}}{f_{n,k}}
R^{j-k}\r)\IND{B_k}} \\
&\leq \P{B_k}\log^2\l(\frac{\EW{|\xi_k|\IND{B_k}}}{\eta\P{B_k}}+\frac{1}{\P{B_k}}\EW{\frac{1}{\eta}\sum_{j=k+1}^\infty \binom{j}{k} |\xi_j| \frac{f_{n,j}}{f_{n,k}} R^{j-k}\IND{B_k}}\r) \\
&\leq \P{B_k}\log^2\l(\frac{\EW{|\xi_k|\IND{\{|\xi_k|>\e\eta\}}\prod_{l=0}^{k-1}\IND{\{|\xi_l|\leq \e \eta\}}}}{\eta q^k(1-q)}+\frac{1}{\eta}\sum_{j=k+1}^\infty \binom{j}{k}  \frac{f_{n,j}}{f_{n,k}} R^{j-k}\r) \\
&\leq \P{B_k}\log^2\l(\frac{\EW{|\xi_k|\IND{\{|\xi_k|>\e\eta\}}}}{\eta (1-q)}+\frac{C}{\eta f_{n,k}}\left(\frac{2}{R\log(A/R)}\right)^{k}\e^{np(\sqrt{AR})/2}\r) \\
&\leq \P{B_k}\log^2\l(C+\frac{C}{\eta f_{n,k}}\left(\frac{2}{R\log(A/R)}\right)^{k}\e^{np(\sqrt{AR})/2}\r).
\end{split}
\end{align}
We have to estimate the right-hand side.
Recall that
\begin{equation*}
f^2_{n,k}=\begin{cases}
\frac{n^k}{k!} & \textrm{in the \FAF\ case,} \\
\binom{n+k-1}{k} & \textrm{in the \HAF\ case}.
\end{cases}
\end{equation*}
In the $\FAF$ case, we can use the inequality $(a+b)^2 \leq 2(a^2+b^2)$ to obtain the estimate
\begin{align}\label{eq:logquadratfnk_est_FAF}
\log^2(f^2_{n,k})
=\left(k\log n-\log (k!)\right)^2
\leq
2k^2\log^2n+2\log^2(k!)
\leq
C(k^4+1) n^2.
\end{align}
In the $\HAF$ case,  we have
\begin{align}\label{eq:logquadratfnk_est_HAF}
\log^2(f^2_{n,k})
=\log^2\l(\binom{n+k-1}{k}\r)
\leq \log^2\l((n+k)!\r)
\leq C(n+k)^4
\leq 8C (n^4 + k^4).
\end{align}
Taking together the estimates \eqref{eq:logquadratfnk_est_FAF} and \eqref{eq:logquadratfnk_est_HAF}
we obtain the following estimate which is valid both for \HAF\ and \FAF:
\begin{equation}\label{eq:logquadratfnk_est}
\log^2(f^2_{n,k}) \leq  C(k^4+1) n^4 \quad \textrm{for all }k\in\N_0, n\in\N.
\end{equation}
For a given $C>1$ there exists a constant $\tilde{C}>0$ such that $\log^2(C+x)\leq \log^2 x+\tilde{C}$ for all $x>0$.
Using this observation together with the estimate~\eqref{eq:logquadratfnk_est} and the inequality $(a +  b + c + d)^2 \leq 4(a^2+b^2+c^2+d^2)$, we can estimate the right-hand side of~\eqref{eq:est_two} as follows:
\begin{align*}
&\P{B_k}\log^2\left(C + \frac{C}{\eta f_{n,k}}\left(\frac{2}{R\log(A/R)}\right)^k\e^{np(\sqrt{AR})/2}\right) \\
&\leq \tilde C \P{B_k}  + \P{B_k} \log^2\left(\frac{C}{\eta f_{n,k}}\left(\frac{2}{R\log(A/R)}\right)^k\e^{np(\sqrt{AR})/2}\right) \\
& = \tilde C \P{B_k}+\P{B_k}\left(\log\frac{C}{\eta}-\frac{1}{2}\log\l(f_{n,k}^2\r)+k\log\l(\frac{2}{R\log(A/R)}\r)+\frac{n}{2}p\l(\sqrt{AR}\r)\right)^2 \\
&\leq \tilde C \P{B_k} + 4\P{B_k}\left(\log^2\frac{C}{\eta}+\frac{1}{4}\log^2\l(f_{n,k}^2\r)+k^2\log^2\l(\frac{2}{R\log(A/R)}\r)+\frac{n^2}{4}p^2\l(\sqrt{AR}\r)\right) \\
&\leq C\P{B_k}(k^4+1) n^4,
\end{align*}
where in the last step we estimated terms of the form $1, k^2, n^2$ by $n^4, k^4n^4, n^4$, respectively.
Combining this with the above estimates \eqref{eq:est_one} and \eqref{eq:est_two} yields
\begin{equation*}
\EW{N_n^\kappa\IND{\l\{N_n \geq \e^{n^\eps}\r\}}\IND{B_k}}
\leq C\e^{-n^\eps(2-\kappa)} \P{B_k}(k^4+1) n^4
=
C\e^{-n^\eps(2-\kappa)} n^4 (k^4+1) q^k (1-q)
.
\end{equation*}
Therefore, taking the sum over $k=0,1,\ldots$, we arrive at
\begin{align*}
\EW{N_n^\kappa\IND{\l\{N_n \geq \e^{n^\eps}\r\}}}
=\sum_{k=0}^\infty \EW{N_n^\kappa\IND{\l\{N_n \geq\e^{n^\eps}\r\}}\IND{B_k}}
\leq C \e^{-n^\eps(2-\kappa)} n^4 \sum_{k=0}^\infty (k^4+1)q^k(1-q).
\end{align*}
The right-hand side  converges to $0$ as $n\to\infty$ because the sum is a finite constant. This completes the proof of~\eqref{eq:N_kappa_ab_leq_C}.
\end{proof}

\subsection{Proof of Lemma~\ref{lemma:abschaetzungdm}}
Fix some sufficiently small $0<\delta<1/2$, some $1<\kappa<2$ and a compact set $K\subset \mathcal{D}\cap (\R\backslash\{0\})$.
Our aim is to prove that there exist constants $C>0$ and $c>0$ such that
\begin{equation}\label{eq:P_D_t_m_n_delta_est}
\P{N_n[t,t+\delta n^{-1/2}] \geq m}\leq C\left(\left(2\delta\right)^{(2/3)m}+\left(2\delta\right)^{-(1/3)m}\EXP{-c n^{1/5}}\right) \quad
\textrm{for all }n,m\in\N, \textrm{ for all } t\in K.
\end{equation}

\begin{proof}[Proof of~\eqref{eq:P_D_t_m_n_delta_est}]
Recall from~\eqref{eq:definitionQ} that the process $Q_{n,t}$ is given by
$$
Q_{n,t} (z) =  \left(v_n\left(t+\frac{z}{\sqrt{n}}\right)\right)^{-1/2} P_n\left(t+\frac{z}{\sqrt{n}}\right).
$$
For $T>0$ and $m\in\N$ we write
\begin{equation}\label{eq:zerlegungPDMJ}
\P{N_n[t,t+\delta n^{-1/2}] \geq m} \leq \P{N_n[t,t+\delta n^{-1/2}] \geq m, \left|Q_{n,t}(0)\right|\geq T} +
\P{\left|Q_{n,t}(0)\right|< T}
\end{equation}
and estimate the terms on the right-hand side separately.

\vspace*{2mm}
\noindent
\textit{The first term on the right-hand side of~\eqref{eq:zerlegungPDMJ}}
will be estimated by using a lemma of Ibragimov and Maslova~\cite{ibragimov_maslova1}; see also~\cite[Lemma~4.4]{flaschekabluchko2017} for the proof in the generality needed here.
In our setting, this lemma states that
\begin{equation*}
\P{N_n[t,t+\delta n^{-1/2}] \geq m, \left|Q_{n,t}(0)\right|\geq T}
\leq \frac{\delta^{2m}}{(m!)^2T^2}\sup_{x\in[0,\delta]} \E \l| Q^{(m)}_{n,t}(x)\r|^2,
\end{equation*}
where $Q^{(m)}_{n,t}$ denotes the $m$-th derivative of $Q_{n,t}$. Since for sufficiently large $n$, the function $(Q_{n,t}(z))_{|z|< 2}$ is analytic,
Cauchy's integral formula for analytic functions yields, for all $x\in [0,\delta]$,
\begin{equation*}
\l|Q^{(m)}_{n,t}(x)\r|=
\frac{m!}{2\pi}
\l|\oint_{\partial \bD} \frac{Q_{n,t}(z)}{(z-x)^{m+1}}\D z \r|
\leq\frac{m!}{2\pi}\int_{\partial \bD} \frac{|Q_{n,t}(z)|}{(1-\delta)^{m+1}}\ |\d z|,
\end{equation*}
where $\partial \bD=\{|z|=1\}$ is the unit circle.
After squaring, taking the expectation and using Jensen's inequality for the quadratic function two times
we obtain
\begin{align*}
\E \l|Q_{n,t}^{(m)}(x)\r|^2
&\leq \frac{(m!)^2}{4\pi^2(1-\delta)^{2m+2}}\E \left(\int_{\partial \bD} |Q_{n,t}(z)|\ |\d z|\right)^2
\\
&\leq \frac{(m!)^2}{2\pi(1-\delta)^{2m+2}}\E \int_{\partial \bD} |Q_{n,t}(z)|^2\ |\d z| \\
&\leq \frac{(m!)^2}{2\pi(1-\delta)^{2m+2}}\int_{\partial \bD} \E |Q_{n,t}(z)|^2\ |\d z| \\
&\leq \frac{(m!)^2}{ (1-\delta)^{2m+2}}\sup_{z\in \partial \bD} \E |Q_{n,t}(z)|^2.
\end{align*}
Since the above holds for arbitrary $x\in [0,\delta]$,  it follows that
\begin{equation*}
\sup_{x\in[0,\delta]} \E \l| Q^{(m)}_{n,t}(x)\r|^2
\leq \frac{(m!)^2}{2\pi (1-\delta)^{2m+2}}\sup_{|z|\leq 1} \E \l|Q_{n,t}(z)\r|^2.
\end{equation*}
Using the estimate given in \eqref{eq:qvarabsch} (with $R=1$) we find a constant $C>0$ such that
\begin{equation*}
\E \left|Q_{n,t}(z)\right|^2\leq C \quad \textrm{for all } |z|\leq 1,\;\;
n\in\N, \;\; t\in K.
\end{equation*}
Taking everything together yields the following estimate for the first term on the right-hand side of \eqref{eq:zerlegungPDMJ}:
\begin{equation}\label{eq:erstertermabsch}
\P{N_n[t,t+\delta n^{-1/2}] \geq m, \left|Q_{n,t}(0)\right|\geq T} \leq
\frac{C}{T^2}\frac{\delta^{2m}}{(1-\delta)^{2(m+1)}}\leq
\frac{C}{T^2}\left(\frac{\delta}{1-\delta}\right)^{2m}
\leq \frac{C}{T^2}\left(2\delta\right)^{2m},
\end{equation}
for every $T>0$, where we used that $0<\delta<1/2$ in the last step.

\vspace*{2mm}
\noindent
\textit{The second term on the right-hand side of~\eqref{eq:zerlegungPDMJ}}
will be estimated in  Lemma \ref{lemma:abschaetzungdmeins}, below, which  states that
$$
\P{\left|Q_{n,t}(0)\right|< T} \leq C\left(T+T^{-1/2} \EXP{-c n^{1/5}}\right).
$$
Taking the estimates for both terms on the right-hand side of~\eqref{eq:zerlegungPDMJ} together, we obtain
\begin{equation*}
\P{N_n[t,t+\delta n^{-1/2}] \geq m}\leq C\left(
\frac{\left(2\delta\right)^{2m}}{T^2}+T+T^{-1/2}\EXP{-c n^{1/5}}
\right).
\end{equation*}
Optimizing the bound by choosing $T=(2\delta)^{(2/3)m}$ completes the proof of~\eqref{eq:P_D_t_m_n_delta_est}.
\end{proof}

\section{Estimating the probability of small values of \texorpdfstring{$P_n$}{Pn}} \label{sec:small_values}
Recall from~\eqref{eq:definitionQ} that
$$
Q_{n,t}(0) = \frac{P_n(t)}{\sqrt{v_n(t)}} = \sum_{k=0}^\infty \frac{f_{n,k} t^k}{\sqrt{v_n(t)}} \xi_k.
$$
The main result of the present section is the following lemma that estimates the probabilities of small values of $Q_{n,t}(0)$.
\begin{lemma} \label{lemma:abschaetzungdmeins}
Fix a compact set $K\subset \mathcal{D}\cap (\R\backslash \{0\})$.
There exist constants $C=C(K)>0$ and $c=c(K)>0$ such that
\begin{equation*}
\P{\left|Q_{n,t}(0)\right|< T} \leq C\left(T+T^{-1/2} \EXP{-c n^{1/5}}\right)
\end{equation*}
for all $T>0$, $n \geq n_0$ and for all $t\in K$.
\end{lemma}
\begin{proof}
The first step is a smoothing argument similar to that used by Ibragimov and Maslova~\cite{ibragimov_maslova1}.
Take some $\lambda>0$ (to be chosen explicitly at the very end of the proof) and consider the random variable
\begin{equation*}
\tilde{Q}_{n}:=Q_{n,t}(0)+\Theta_\lambda,
\end{equation*}
where $\Theta_\lambda$ is the sum of two independent random variables that are uniformly distributed on
the interval $[-\lambda,\lambda]$ and independent of $Q_{n,t}(0)$. The characteristic function of $\Theta_\lambda$ is given
by
\begin{equation*}
\psi_\lambda(u):=\EW{\EXP{\I u \Theta_\lambda}}=\frac{\sin^2(u\lambda)}{u^2\lambda^2}.
\end{equation*}
For $T>0$ we have the estimate
\begin{equation}\label{eq:zerlegunglemmazehn}
\P{\left|Q_{n,t}(0)\right|\leq T}\leq \P{\l|\tilde{Q}_{n}\r| \leq \frac{3}{2}T}
+\P{|\Theta_\lambda|\geq \frac{1}{2}T}.
\end{equation}
In the following we shall estimate the two terms on the right-hand side of \eqref{eq:zerlegunglemmazehn}.

\vspace*{2mm}
\noindent
\textit{Second term of on the right-hand side of \eqref{eq:zerlegunglemmazehn}.} By Chebyshev's inequality,
\begin{equation}\label{eq:estimate_theta_cheby}
\P{\l|\Theta_\lambda\r|\geq \frac{1}{2}T}\leq \frac{4\lambda^2}{3T^2}.
\end{equation}

\vspace*{2mm}
\noindent
\textit{First term on the right-hand side of \eqref{eq:zerlegunglemmazehn}.}
Let $\tilde{\phi}_{n}$ denote the characteristic function of $\tilde{Q}_{n}$, that is
$$
\tilde{\phi}_{n}(u) = \psi_\lambda(u) \prod_{k=0}^\infty \phi(a_{n,k} u),
$$
where $\phi(u):=\E \EXP{\I u\xi_0}$ is the characteristic function of $\xi_0$ and
\begin{equation*}
a_{n,k}:=\frac{f_{n,k}
t^k}{\sqrt{v_n(t)}}.
\end{equation*}
The density of $\tilde{Q}_{n}$ exists because the random variable $\Theta_\lambda$ is absolutely continuous. Using the inversion formula for  the Fourier transform, the
distribution function of
$|\tilde{Q}_{n}|$ can be written as
\begin{align}
\begin{split}\label{eq:P_Q_leq_y}
\P{\l|\tilde{Q}_{n}\r|\leq y}
&=\frac{1}{2\pi}\INT{-y}{y}\INT{-\infty}{\infty} \tilde{\phi}_{n}(u) \e^{-\I u x}\D u \D x
=\frac{1}{\pi} \INT{0}{\infty} \frac{\sin(yu)}{u}
\RE \tilde{\phi}_{n}(u) \D u\\
&\leq
\frac{y}{\pi} \INT{0}{\infty} \psi_\lambda(u)
\prod_{k=0}^\infty\left|\phi\left(a_{n,k} u\right)\right|
\D u
\end{split}
\end{align}
for all $y\geq 0$, where in the last inequality we used the bounds $|\sin(yu)|\leq y u$ and $|\text{Re}\, \tilde{\phi}_{n}(u)| \leq | \tilde{\phi}_{n}(u)|$.

Since $\sum_{k=0}^\infty a_{n,k}^2=1$ for all $n\in\N$, we can view $\{a^2_{n,k}:k\in\N_0\}$ as a discrete probability distribution on $\N_0$.
In fact, in our four special cases of interest this distribution is given by
\begin{itemize}
\item the binomial distribution  $\textrm{Bin}(n,q(t))$ in the \SP\ case,
\item the Poisson distribution  $\textrm{Poi}(n q(t))$ in the \FAF\ case,
\item the negative binomial\footnote{A random variable $Z$ has negative binomial distribution $\textrm{NBin}(n,p)$ if $\P{Z=k} = \binom{n+k-1}{k} p^n(1-p)^k$ for $k=0,1,\ldots$.} distribution  $\textrm{NBin}(n,q(t))$ in the \HAF\ case,
\item the Poisson distribution  $\textrm{Poi}(n q(t))$ conditioned to the interval  $\{0,\ldots,n\}$ in the \WP\ case,
\end{itemize}
where
\begin{equation}\label{eq:def_q(t)}
q(t)
:=
\begin{cases}
t^2/(1+t^2) &\textrm{in the \SP\ case}, \\
t^2 &\textrm{in the \WP\ and the \FAF\ case,} \\
1-t^2 &\textrm{in the \HAF\ case}.
\end{cases}
\end{equation}
Indeed, in the \SP\ case we have
\begin{equation}\label{eq:a_n_k_binomial}
a_{n,k}^2 = \frac{f_{n,k}^2 t^{2k}}{v_n(t)}
=
\binom{n}{k}\left(\frac{t^2}{1+t^2}\right)^k\left(\frac{1}{1+t^2}\right)^{n-k}
=
\binom{n}{k}q(t)^k(1-q(t))^{n-k}.
\end{equation}
Similarly, in the \HAF\ case we have
\begin{equation}\label{eq:a_n_k_neg_binomial}
a_{n,k}^2 = \frac{f_{n,k}^2 t^{2k}}{v_n(t)}
=
\binom{n+k-1}{k}(1-t^2)^nt^{2k}
=
\binom{n+k-1}{k}q(t)^n (1-q(t))^{k}.
\end{equation}
Finally, in the \FAF\ case,
\begin{equation}\label{eq:a_n_k_poi}
a_{n,k}^2 = \frac{f_{n,k}^2 t^{2k}}{v_n(t)}
=
\e^{-nt^2} \frac{(nt^2)^k}{k!}.
\end{equation}
The \WP\ case is similar to the \FAF\ case except that now we have the restriction $k\in \{0,\ldots,n\}$ and the definition of $v_n(t)$ should be modified.
Observe also that  as long as  $t\in K$, $q(t)$ is bounded away from $0$ (in all four cases) and from $1$ (in the binomial and negative binomial cases).

The aforementioned discrete distributions are unimodal and their mode will be denoted by
\begin{equation*}
m_n :=\textrm{arg max}_{k=0,1,\ldots} a^2_{n,k}.
\end{equation*}
If there are several modes, we agree to take the smallest one. Since the random variables $\xi_0,\xi_1,\ldots$ are supposed to have zero mean and unit variance,
their characteristic function $\phi(u)=\E\EXP{\I u \xi_0}$ can be estimated by
\begin{equation}\label{eq:nontrivialestimationofcharfunction}
|\phi(u)|\leq \EXP{-\frac{u^2}{4}} \quad \textrm{for all } u\in[-\tilde{c},\tilde{c}]
\end{equation}
provided  $\tilde{c}>0$ is sufficiently small.
Let us cover $[0,\infty)$ by the following intervals:
\begin{align*}
&\Gamma_{n,0}:= \l[0, \tilde{c}/a_{n,m_n}\r), \quad \Gamma_{n,\infty}:= \l[\tilde{c}/a_{n,0},\infty\r),\\
&\Gamma_{n,k}:=
[\tilde{c}/ a_{n,k},\tilde{c}/ a_{n,k-1}), \quad \textrm{ for } k=1,\ldots, m_n,
\end{align*}
and define
\begin{align}\label{eq:def_Ink}
I_{n,k}:=
\int_{\Gamma_{n,k}} \psi_\lambda(u)\prod_{j=0}^\infty |\phi(a_{n,j} u)|\D u.
\end{align}
With this notation,  we can write~\eqref{eq:P_Q_leq_y} as follows:
\begin{equation}\label{eq:P_Q_leq_y_2}
\P{\l|\tilde{Q}_{n}\r|\leq y}  \leq  \frac y \pi \left(I_{n,0} + I_{n,\infty} + \sum_{k=1}^{m_n} I_{n,k}\right).
\end{equation}
In the following, we shall estimate the integrals $I_{n,k}$.

\vspace*{2mm}
\noindent
\textit{Estimate for $k=0$.}
For $u\in\Gamma_{n,0}$ we have $\max_{j=0,1,\ldots} |a_{n,j} u|  \leq \tilde{c}$ and therefore we can estimate all factors $|\phi(a_{n,j}u)|$
 in the nontrivial way by using~\eqref{eq:nontrivialestimationofcharfunction}, thus arriving at
\begin{equation}\label{eq:I_n_k_0}
I_{n,0} = \int_{\Gamma_{n,0}}\psi_{\lambda}(u)\prod_{j=0}^\infty |\phi(a_{n,j} u)| \D u
\leq\INT{0}{\tilde{c}/a_{n,m_n}}\EXP{-\frac{u^2}{4}\sum_{j=0}^\infty a_{n,j}^2} \D u
\leq\INT{0}{\infty}\EXP{-\frac{u^2}{4}} \D u \leq C,
\end{equation}
where we also used the estimate $|\psi_{\lambda}(u)|\leq 1$.

\vspace*{2mm}
\noindent
\textit{Estimate for $k = \infty$.}
On the interval $\Gamma_{n,\infty}$ we can use only the trivial estimate $|\phi(a_{n,j}u)|\leq 1$, which gives
\begin{equation}\label{eq:I_n_k_infty}
I_{n,\infty}
=\int_{\tilde{c}/a_{n,0}}^\infty \psi_{\lambda}(u)\prod_{j=0}^\infty |\phi(a_{n,j} u)| \D u
\leq
\int_{\tilde{c}/a_{n,0}}^\infty \frac 1 {u^2\lambda^2} \D u
=
\frac {a_{n,0}}{\tilde {c}\lambda^2}=\frac {f_{n,0}}{\tilde {c}\lambda^2 \sqrt{v_n(t)}}
=
O(\eee^{-n\eps})
\end{equation}
for sufficiently small $\eps>0$, where the last step follows from $f_{n,0}=1$ and the formula for $v_n(t)$; see~\eqref{eq:vn_explicit_fourcases}.

\vspace*{2mm}
\noindent
\textit{Estimate for $k\in \{1,\ldots,m_n\}$.}
On the interval $\Gamma_{n,k}$,  we are able to estimate the first $k$ factors of the product $\prod_{j=0}^\infty |\phi(a_{n,j}u)|$ non-trivially by~\eqref{eq:nontrivialestimationofcharfunction}, while the
remaining factors must be estimated trivially by $|\phi(a_{n,j}u)|\leq 1$.
It is convenient to introduce the partial sums of $a_{n,j}^2$ as follows:
\begin{equation*}
F_{n,k}:=\sum_{j=0}^{k-1} a_{n,j}^2.
\end{equation*}
For any $k\in \{1,\ldots, m_n\}$, the following estimate holds:
\begin{align*}
I_{n,k}&\leq\INT{\tilde{c}/a_{n,k}}{\tilde{c}/a_{n,k-1}} \frac{1}{u^2
\lambda^2}\EXP{-\frac{u^2}{4}F_{n,k}} \D u
\leq\INT{\tilde{c}\sqrt{F_{n,k}}/a_{n,k}}{\infty}\frac{\sqrt{F_{n,k}}}{s^2\lambda^2}\EXP{-\frac{s^2}{4}}\D s \\
&\leq C \frac{a_{n,k}}{\lambda^2}\EXP{-\nu\frac{F_{n,k}}{a_{n,k}^2}}
\end{align*}
for some sufficiently small $\nu>0$, where we used the inequality $\int_x^\infty s^{-2} \e^{-s^2/4} \D s \leq C x^{-1} \e^{-x^2/8}$, for all $x>0$ and some sufficiently large $C$.
Taking the sum, we may write
\begin{equation}\label{eq:sum_I}
\sum_{k=1}^{m_n} I_{n,k}\leq \frac{C}{\lambda^2}\l(
\sum_{k=1}^{\lfloor m_n-n^{1/10}\sqrt{n}\rfloor} a_{n,k}
+\sum_{k=\lceil m_n-n^{1/10}\sqrt{n}\rceil}^{m_n} \EXP{-\nu\frac{F_{n,k}}{a_{n,k}^2}}\right).
\end{equation}
%
In the following, we shall estimate both sums on the right-hand side of~\eqref{eq:sum_I}. But first we need to introduce some notation.
Let $X_1,X_2,\ldots$ be a sequence of i.i.d.\ random variables with
\begin{equation*}
X_{1}\sim \begin{cases}
\BINOMIALVERT{1}{q(t)} & \textrm{in the \SP\ case}, \\
\POI{q(t)} & \textrm{in the \FAF\ or \WP\ case},\\
\NBINOMIALVERT{1}{q(t)} & \textrm{in the \HAF\ case},
\end{cases}
\end{equation*}
where $q(t)$ is defined by~\eqref{eq:def_q(t)}.
Consider their partial sums $S_{n}:=\sum_{j=1}^n X_{j}$. The convolution properties of the above three distributions combined with~\eqref{eq:a_n_k_binomial}, \eqref{eq:a_n_k_neg_binomial}, \eqref{eq:a_n_k_poi} imply that  for all $k\in\N_0$ we have
\begin{equation*}
a_{n,k}^2
=
\begin{cases}
\P{S_{n}=k} & \textrm{in the \SP, \HAF\ and \FAF\ cases},\\
\BP{S_{n}=k}{S_{n}\leq n} & \textrm{in the \WP\ case},
\end{cases}
\end{equation*}
and, consequently,
$$
F_{n,k}
=
\begin{cases}
\P{S_{n} < k} & \textrm{in the \SP, \HAF\ and \FAF\ cases},\\
\BP{S_{n}<k}{S_{n}\leq n} & \textrm{in the \WP\ case}.
\end{cases}
$$

In the following, we shall estimate $a_{n,k}^2$ and $F_{n,k}$ using various refinements of the central limit theorem. Alternatively, the same could be done by using the Stirling formula.
Define $\mu=\mu(t)$ and $\sigma=\sigma(t)$ to be the mean and the variance of $X_1$, namely
\begin{align*}
\mu
:=\E X_{1}
&=\begin{cases}
q(t)  &\textrm{in the \SP, \FAF\ and \WP\ cases},\\
(1-q(t))/q(t) &\textrm{in the \HAF\ case},
\end{cases}\\
\sigma^2 := \VAR X_{1}
&=
\begin{cases}
q(t)(1-q(t))& \textrm{in the \SP\ case,} \\
q(t) &\textrm{in the \FAF\ and \WP\ cases},\\
(1-q(t))/q^2(t)& \textrm{in the \HAF\ case}.
\end{cases}
\end{align*}
Observe that the functions $t\mapsto \mu(t)$ and $t\mapsto \sigma^2(t)$ are continuous and do not vanish on the compact set $K\subset \mathcal{D}\cap (\R\backslash\{0\})$.  Thus, both functions are bounded away from $0$ and $\infty$, which we shall repeatedly use in the sequel.


\vspace*{2mm}
\noindent
\textit{First sum on the right-hand side of~\eqref{eq:sum_I}.}
There exists a constant $\tilde{\gamma}>0$ such that in all four cases of interest, the mode $m_n$ satisfies $|m_n - n\mu| \leq \tilde{\gamma}$. Indeed, the (smallest) mode of the Poisson distribution with mean $nq(t)$ is given by $\lceil n q(t) \rceil -1$. Similarly, for the binomial distribution with parameters $(n, q(t))$, the mode is $\lceil (n+1) q(t)\rceil -1$, while for the negative binomial distribution with parameters $(n,q(t))$ it is $\lceil (1-q(t))(n-1)/q(t)\rceil-1$.
This observation on the mode $m_n$ yields
\begin{equation}\label{eq:modusersetzen}
F_{n,\lceil m_n-n^{1/10}\sqrt{n}\rceil} \leq F_{n,\lceil n\mu + \tilde{\gamma}-n^{1/10}\sqrt{n}\rceil}.
\end{equation}
By Theorem 1 of Chapter 14, Section 6 of~\cite{feller1966} (which states that the central limit theorem holds in the sense of asymptotic equivalence provided the standardized deviation from the mean does not exceed $o(n^{1/6})$), we obtain in all cases except the \WP,
\begin{align*}
F_{n, \lceil n\mu+\tilde{\gamma}-n^{1/10}\sqrt{n}\rceil}
&=\P{S_{n} \leq \lfloor n\mu+\tilde{\gamma}-n^{1/10}\sqrt{n}\rfloor}
=\P{\frac{S_{n}-n\mu}{\sigma\sqrt{n}} \leq \frac{\lfloor n\mu+\tilde{\gamma}-n^{1/10}\sqrt{n}\rfloor - n\mu}{\sigma \sqrt n}}\\
&\sim\Phi\l(-\frac{n^{1/10}}{\sigma} + O\left(\frac 1 {\sigma \sqrt n}\right)\r)
\leq C\EXP{- n^{1/5}/C},
\end{align*}
as $n\to\infty$, where $\Phi$ is the standard normal distribution function satisfying $\Phi(z) \sim \eee^{-z^2/2}/(\sqrt{2\pi}|z|)$ as $z\to -\infty$. In the last step we used that the function $t\mapsto\sigma(t)$ is continuous and hence bounded on $K$.
In the \WP\ case we have
$$
F_{n, \lceil n\mu+\tilde{\gamma}-n^{1/10}\sqrt{n}\rceil}
=
\frac{\P{S_{n} \leq \lfloor n\mu+\tilde{\gamma}-n^{1/10}\sqrt{n}\rfloor}}{\P{S_n\leq n}}.
$$
Since $\mu(t) = t^2<1$ stays bounded away from $1$ for $t\in K$, $\P{S_n\leq n}$ converges to $1$ by the law of large numbers, and the same argument as in the \FAF\ case applies.
Therefore, for sufficiently small $c>0$, in all four cases it holds that
\begin{equation}\label{eq:zerlegung_Iklambdaeinseins}
\sum_{k=1}^{\lfloor m_n-n^{1/10}\sqrt{n}\rfloor} a_{n,k}
\leq
m_n \sqrt{F_{n, \lceil m_n-n^{1/10}\sqrt{n}\rceil}}
\leq
C m_n\EXP{- n^{1/5}/(2C)}
\leq C\EXP{-c n^{1/5}}.
\end{equation}

\vspace*{2mm}
\noindent
\textit{Second sum on the right-hand side of~\eqref{eq:sum_I}.}
Let us now estimate $F_{n,k}$ and $a_{n,k}^2$ uniformly over the range $k\in \{\lceil m_n-n^{1/10}\sqrt{n}\rceil,\dots, m_n\}$. Again, consider any of the four models except the \WP.
By Theorem 1 of Chapter 14, Section 6 of~\cite{feller1966}, we have
\begin{align*}
F_{n,k} = \P{S_{n} < k} \sim \Phi\left(\frac{k-n\mu}{\sigma\sqrt{n}}\right) = \Phi(x_{n,k}), \quad \textrm{as }n\to\infty,
\end{align*}
where we used the abbreviation
\begin{align*}
x_{n,k}&:=(k-n\mu)/(\sigma \sqrt{n})
\end{align*}
and the fact that $x_{n,k} = o(n^{1/6})$ in the aforementioned range of $k$.
To estimate $a_{n,k}^2$, we use a refined form of the local limit theorem. By Theorem 13 on p.\ 205 in~\cite{petrov}, we have
\begin{equation}\label{eq:asymptoticank}
\P{S_{n}=k}=\frac{1}{\sigma\sqrt{2\pi n}}\EXP{-\frac{x^2_{n,k}}{2}}\left(1+\frac{1}{\sqrt{n}}\frac{\kappa_3}{6\sigma^3}\left(x^3_{n,k}-3x_{n,k}\right) + o\left(\frac{1}{\sqrt{n}}\right)\right),
\end{equation}
uniformly over $k\in\Z$, where $\kappa_3$ is the third cumulant of $X_{1}$.
For  $k\in \{\lceil m_{n}-n^{1/10}\sqrt{n}\rceil,\dots, m_n\}$ we have $x_{n,k} = o(n^{1/6})$ and therefore the above simplifies to
\begin{equation*}
a_{n,k}^2 = \P{S_{n}=k} \sim \frac{1}{\sigma\sqrt{2\pi n}}\EXP{-\frac{x^2_{n,k}}{2}}, \quad \textrm{as }n\to\infty.
\end{equation*}
Using the inequality
$\Phi(x)\geq \frac{1}{\sqrt{2\pi}} \frac{1}{1+x^2}\eee^{-x^2/2}$, $x\in \R$,
we obtain
$$
\frac{F_{n,k}}{a^2_{n,k}}
\geq
\frac {c_1} {1+x_{n,k}^2} \eee^{-x_{n,k}^2/2} \cdot \sigma\sqrt{n} \, \eee^{x_{n,k}^2/2}
=
\frac{c_1 \sigma \sqrt n}{1+ x_{n,k}^2}
\geq
c_2 n^{2/5}.
$$
This leads to the estimate
\begin{equation}\label{eq:zerlegung_Iklambdaeinseins1}
\sum_{k=\lceil m_n-n^{1/10}\sqrt{n} \rceil}^{m_n} \EXP{-\nu\frac{F_{n,k}(t)}{a^2_{n,k}(t)}}
\leq
m_{n} \EXP{- c_2 n^{2/5}} \leq  C\EXP{-c n^{1/5}}
\end{equation}
that is valid in the \SP, \FAF\ and \HAF\ cases. In the \WP\ case, we have
$$
\frac{F_{n,k}}{a_{n,k}^2}
=
\frac{\BP{S_n < k}{S_n\leq n}}{\BP{S_n = k}{S_n\leq n}}
=
\frac{\P{S_n < k}}{\P{S_n = k}},
\quad k=0,\ldots,n,
$$
so that the same argument as above applies, thus establishing~\eqref{eq:zerlegung_Iklambdaeinseins1} in the \WP\ case.

In any of the four cases, recalling~\eqref{eq:sum_I} and using the estimates~\eqref{eq:zerlegung_Iklambdaeinseins} and~\eqref{eq:zerlegung_Iklambdaeinseins1}, we arrive at
\begin{equation}\label{eq:I_n_k_all}
\sum_{k=1}^{m_{n}} I_{n,k} \leq \frac{C}{\lambda^2}\EXP{-c n^{1/5}}.
\end{equation}

\vspace*{2mm}
\noindent
\textit{Completing the proof of Lemma~\ref{lemma:abschaetzungdmeins}.}
Taking the estimates~\eqref{eq:I_n_k_0}, \eqref{eq:I_n_k_infty} and~\eqref{eq:I_n_k_all} together and plugging them into~\eqref{eq:P_Q_leq_y_2} yields
\begin{equation*}
\P{\l|\tilde Q_{n}\r|\leq y} \leq Cy\left(1+
\frac{1}{\lambda^2}\EXP{-c n^{1/5}}\right).
\end{equation*}
Recalling~\eqref{eq:zerlegunglemmazehn} and~\eqref{eq:estimate_theta_cheby} and taking $y=3T/2$ we arrive at
\begin{equation*}
\P{\l|Q_{n,t}(0)\r|\leq T} \leq C\left(T+
\frac{T}{\lambda^2}\EXP{-c n^{1/5}}+\frac{\lambda^2}{T^2}\right).
\end{equation*}
Choosing $\lambda^2=T^{3/2}\EXP{-c n^{1/5}}$ optimizes this bound and completes the proof of Lemma~\ref{lemma:abschaetzungdmeins}.
\end{proof}

\section{Auxiliary lemmas}\label{sec:auxiliary}

In the following lemma we prove a statement which was used in the proof of Theorem~\ref{theo:wkonvgauss} when verifying the Lindeberg condition. Recall from~\eqref{eq:a_n_k_def} that we defined
$$
a_{n,k}(z):=\frac{f_{n,k}
\left(t+\frac{z}{\sqrt{n}}\right)^k}{\sqrt{v_n(t+\frac{z}{\sqrt{n}})}}.
$$

\begin{lemma}\label{lemma:locallimittheorem}
Fix some $t\in \mathcal{D}\cap (\R\backslash \{0\})$. Then, for every $z\in\CN$ we have
\begin{equation*}
\lim_{n\to\infty} \max_{k=0,1,\ldots}|a_{n,k}(z)|^2=0.
\end{equation*}
\end{lemma}
\begin{proof}
Let $s_n = t + \frac{z}{\sqrt n}$. Then,
\begin{equation}\label{eq:a_n_k_z_est}
|a_{n,k}(z)|^2 = \frac{f_{n,k}^2 |s_n|^{2k}}{\sqrt{v_n(s_n)v_n(\bar s_n)}}
=
\frac{f_{n,k}^2 |s_n|^{2k}}{v_n(|s_n|)} \cdot \frac{v_n(|s_n|)}{\sqrt{v_n(s_n)v_n(\bar s_n)}}
\leq
C
\frac{f_{n,k}^2 |s_n|^{2k}}{v_n(|s_n|)}
\end{equation}
because by~\eqref{eq:E_Q_n_t_Q_n_t} and \eqref{eq:covariancestructureeins} (where we take $z_i=z$, $z_j=\bar z$),
$$
\frac{v_n(|s_n|)}{\sqrt{v_n(s_n)v_n(\bar s_n)}} = \E [Q_{n,t}(z) Q_{n,t}(\bar z)]  \ton \EXP{2 \gamma(t)(\IM z)^2}.
$$
In the following we shall use a probabilistic interpretation of the right-hand side of~\eqref{eq:a_n_k_z_est}. For each $n\in\N$ let $X_{n,1},\ldots, X_{n,n}$ be i.i.d.\ random variables with
\begin{equation*}
X_{n,1}\sim \begin{cases}
\BINOMIALVERT{1}{p_n} & \textrm{in the \SP\ case},\\
\POI{p_n} & \textrm{in the \FAF\ or \WP\ case},\\
\NBINOMIALVERT{1}{p_n} & \textrm{in the \HAF\ case},
\end{cases}
\end{equation*}
where
$$
p_n =
\begin{cases}
|s_n|^2/(1+|s_n|^2) & \textrm{in the \SP\ case},\\
|s_n|^2 & \textrm{in the \FAF\ or \WP\ case},\\
1-|s_n|^2 & \textrm{in the \HAF\ case}.
\end{cases}
$$
Define their sum $S_n = X_{n,1}+\ldots+X_{n,n}$. Then, recalling the definition of $f_{n,k}$ in~\eqref{eq:thefourcases} and using the convolution property of the binomial, Poisson and negative binomial distributions, we can write
$$
\frac{f_{n,k}^2 |s_n|^{2k}}{v_n(|s_n|)}
=
\begin{cases}
\P{S_{n}=k} & \textrm{in the \SP, \HAF\ and \FAF\ cases},\\
\BP{S_{n}=k}{S_{n}\leq n} & \textrm{in the \WP\ case},
\end{cases}
$$
for all $k\in\N_0$. The computations are analogous to those in the proof of Lemma~\ref{lemma:abschaetzungdmeins}; see~\eqref{eq:a_n_k_binomial}, \eqref{eq:a_n_k_neg_binomial}, \eqref{eq:a_n_k_poi}.
Note that in the \WP\ case, $\lim_{n\to\infty} p_n = t^2 \in (0,1)$, so that $\lim_{n\to\infty} \P{S_n \leq n} = 1$ by the law of large numbers for the Poisson distribution. Therefore, in all four cases it suffices to prove that
\begin{equation}\label{eq:concentration_to_0}
\lim_{n\to\infty} \max_{k=0,1,\ldots} \P{S_n = k}  = 0.
\end{equation}
We shall do this by utilizing the Kolmogorov--Rogozin inequality; see Eq.~(A) on  page~290 of~\cite{Esseen1968}. For a real-valued random variable $Y$ let
\begin{equation}
Q(Y;r):=\sup_{x\in\R} \P{Y\in{[x,x+r]}},\quad r>0,
\end{equation}
denote the concentration function of $Y$. For all three distributions of interest it is easy to check that
\begin{equation}\label{eq:Qwegbeschraenkt}
Q\left(X_{n,1};\frac{1}{2}\right) = \max_{k=0,1,\ldots} \P{X_{n,1} = k} < 1-\eps
\end{equation}
for some $\eps>0$. Indeed, in the  binomial case $X_{n,1}\sim \BINOMIALVERT{1}{p_n}$ we have
\begin{equation*}
Q\l(X_{n,1};\frac{1}{2}\r) = \max  \{p_n,1-p_n\} \ton \max\{p,1-p\} \text{ with } p = \frac{t^2}{1+t^2} \in (0,1).
\end{equation*}
In the negative binomial case $X_{n,1}\sim \NBINOMIALVERT{1}{p_n}$ we have
\begin{equation*}
Q\l(X_{n,1};\frac{1}{2}\r) = p_n \ton 1-t^2 \in (0,1).
\end{equation*}
Finally, in the Poisson case  we have
$$
Q\l(X_{n,1};\frac{1}{2}\r) = \max_{k=0,1,\ldots} \e^{-|s_n|^2} \frac{|s_n|^{2k}}{k!} \ton \max_{k=0,1,\ldots} \e^{-t^2} \frac{t^{2k}}{k!} < 1
$$
since $t\neq 0$.   In all four cases, the Kolmogorov--Rogozin inequality (see Corollary~1 on page~304 of~\cite{Esseen1968}) yields
\begin{equation*}
\max_{k=0,1,\ldots} \P{S_n=k}
=
Q\l(S_n;\frac{1}{2}\r)
\leq C\left(\sum_{k=1}^n \l(1-Q\l(X_{n,k};\frac{1}{2}\r)\r)\right)^{-1/2}
\leq \frac{C'}{\sqrt n},
\end{equation*}
thus proving~\eqref{eq:concentration_to_0} and Lemma~\ref{lemma:locallimittheorem}.
\end{proof}


The following lemma was used in the proof of Lemma~\ref{lemma:manyzeroesareunlikely}.
\begin{lemma}\label{lem:est_det_poly}
Let the $f_{n,k}$'s (and the corresponding set $\mathcal{D}$) be chosen as in~\eqref{eq:thefourcases} (respectively, \eqref{eq:p_corresponding_function}).
Fix $R\in\mathcal{D}\cap \R, R>0$.
Then there exist constants $C>0$ and $A>R, A\in\mathcal{D}\cap\R$, such that
\begin{equation*}
\sum_{j=k}^\infty \binom{j}{k} f_{n,j} R^{j} \leq C \left(\frac{2}{\log(A/R)}\right)^{k} \EXP{\frac{n}{2}p\left(\sqrt{AR}\right)} \quad
\textrm{for all }n\in\N, \;\; k\in\N_0,
\end{equation*}
where $p:\mathcal{D}\to\CN$ is given by \eqref{eq:p_corresponding_function}.
\end{lemma}
\begin{proof}
Choose some $A\in \mathcal{D}\cap\R$ such that  $R < A < \e^2 R$. The Cauchy-Schwarz inequality in the $\ell^2$-space yields
\begin{equation}\label{eq:est_cauchyschwarz_lzwei}
\sum_{j=k}^\infty \binom{j}{k} f_{n,j} R^{j}
\leq \sqrt{\sum_{j=k}^\infty f_{n,j}^2 (AR)^j} \sqrt{\sum_{j=k}^\infty \binom{j}{k}^2 \left(\frac{R}{A}\right)^j}.
\end{equation}

\vspace*{2mm}
\noindent
\textit{First factor on the right-hand side of \eqref{eq:est_cauchyschwarz_lzwei}.}
The assumption on $v_n$ stated in~\eqref{eq:annahmevarmodphi} yields the estimate
\begin{equation}\label{eq:est_firstfactor}
\sqrt{\sum_{j=k}^\infty f_{n,j}^2 (AR)^j} \leq  \sqrt{v_n(\sqrt{AR})} \leq C \e^{np(\sqrt{AR})/2},
\end{equation}
where the function $p:\mathcal{D}\to\R$ is given by \eqref{eq:p_corresponding_function}.

\vspace*{2mm}
\noindent
\textit{Second factor on the right-hand side of \eqref{eq:est_cauchyschwarz_lzwei}.}
The function $x\mapsto x^{2k}(R/A)^x$ attains its maximum at
\begin{equation*}
\argmax_{x\in [0,\infty)} x^{2k}\left(\frac{R}{A}\right)^x =\frac{2k}{\log(A/R)}=:x_k \geq k.
\end{equation*}
Using the inequality $\binom{n}{k}\leq n^k/k!$ for all $n\geq k$ and splitting the sum at this maximum yields
\begin{align}\label{eq:series_k}
\begin{split}
\sum_{j=k}^\infty \binom{j}{k}^2\left(\frac{R}{A}\right)^j
&\leq \frac{1}{(k!)^2}\sum_{j=k}^\infty j^{2k}\left(\frac{R}{A}\right)^j \\
&= \frac{1}{(k!)^2}\left(\sum_{j=k}^{\lceil x_k\rceil} j^{2k} \left(\frac{R}{A}\right)^j
+\sum_{j=\lceil x_k\rceil+1}^{\infty} j^{2k} \left(\frac{R}{A}\right)^j \right).
\end{split}
\end{align}
In the first sum in \eqref{eq:series_k} we estimate every term by the maximum and then use the Stirling formula:
\begin{align}\label{eq:est_poly_sumbymax}
\frac{1}{(k!)^2}\sum_{j=k}^{\lceil x_k \rceil} j^{2k} \left(\frac{R}{A}\right)^j
\leq \lceil x_k \rceil \frac{x_k^{2k}}{(k!)^2} \left(\frac{R}{A}\right)^{x_k}
\leq \left(\frac{2}{\log(A/R)}\right)^{2k+1}\frac{k^{2k+1}}{(k!)^2}\e^{-2k}
\leq C\left(\frac{2}{\log(A/R)}\right)^{2k}.
\end{align}
The sequence $(j^{2k}(R/A)^j)_{j\geq \lceil x_k\rceil}$ is monotone decreasing and thus the second
term of the sum in \eqref{eq:series_k} can be estimated by the corresponding integral:
\begin{align}\label{eq:est_poly_sumbyint}
\begin{split}
\frac{1}{(k!)^2}\sum_{j=\lceil x_k \rceil+1}^\infty j^{2k} \left(\frac{R}{A}\right)^j
&\leq \frac{1}{(k!)^2}\INT{x_k}{\infty} x^{2k} \left(\frac{R}{A}\right)^x \D x
=\frac{1}{(k!)^2}\INT{x_k}{\infty} x^{2k} \EXP{x\log\left(\frac{R}{A}\right)} \D x \\
&=\left(\log\left(\frac{A}{R}\right)\right)^{-(2k+1)}\frac{1}{(k!)^2}\INT{2k}{\infty} u^{2k}\e^{-u}\D u \\
&\leq \left(\log\left(\frac{A}{R}\right)\right)^{-(2k+1)}\frac{\Gamma(2k+1)}{(k!)^2}
\leq C \left(\frac{2}{\log\left(A/R\right)}\right)^{2k}.
\end{split}
\end{align}
Combining the results of \eqref{eq:est_poly_sumbymax} and \eqref{eq:est_poly_sumbyint} with \eqref{eq:series_k} we arrive at
\begin{equation}\label{eq:est_secondfactor}
\sqrt{\sum_{j=k}^\infty \binom{j}{k}^2 \left(\frac{R}{A}\right)^j} \leq C \left(\frac{2}{\log(A/R)}\right)^k.
\end{equation}
Taking~\eqref{eq:est_cauchyschwarz_lzwei}, \eqref{eq:est_firstfactor} and~\eqref{eq:est_secondfactor} together completes the proof.
\end{proof}

\subsection*{Acknowledgement}
The support by the SFB 878 ``Groups, Geometry and Actions'' is gratefully acknowledged. 

\bibliographystyle{plainnat}
\bibliography{weyl_poly_bib}

\end{document}